\documentclass[11pt]{article}

\usepackage[backend=biber]{biblatex}
\usepackage[dvipdfmx]{graphicx}
\usepackage{color}
\usepackage{bm}
\usepackage{amsmath,amsthm,amssymb,url}
\usepackage{multirow,bigdelim}
\usepackage{enumerate}
\usepackage[normalem]{ulem}

\newtheorem{thm}{Theorem}[section]
\newtheorem{theorem}[thm]{Theorem}

\newtheorem{lemma}[thm]{Lemma}

\newtheorem{claim}[thm]{Claim}
\newtheorem{corollary}[thm]{Corollary}

\newtheorem{conjecture}[thm]{Conjecture}

\theoremstyle{definition}
\newtheorem{example}{Example}

\DeclareMathOperator{\rank}{rank}

\newcommand{\lin}{\mathrm{lin}}
\newcommand{\aff}{\mathrm{aff}}

\newcommand{\ant}{\mathrm{Ast}}
\newcommand{\sm}{\setminus}

\newcommand{\cT}{{\cal T}}
\newcommand{\cS}{{\cal S}}
\newcommand{\cU}{{\cal U}}
\newcommand{\cP}{{\cal P}}
\newcommand{\cR}{{\cal R}}
\newcommand{\cV}{{\cal V}}
\newcommand{\cW}{{\cal W}}

\newcommand{\cB}{{\cal B}}

\newcommand{\Z}{{\mathbb Z}}
\newcommand{\R}{{\mathbb R}}

\newcommand{\csir}[1]{$\mathbb Z_2$-irreducible #1-cycle}

\title{Rigidity of Symmetric Simplicial Complexes and the Lower Bound Theorem}

\author{James Cruickshank, Bill Jackson, Shinichi Tanigawa}

\begin{document}

\maketitle

\begin{abstract} 
We show that, if $\Gamma$ is a point group of $\mathbb{R}^{k+1}$ of order two for some $k\geq 2$ and $\cS$ is  $k$-pseudomanifold which has a free automorphism of order two, then either $\cS$ has a $\Gamma$-symmetric infinitesimally rigid realisation in $\R^{k+1}$ or $k=2$ and $\Gamma$ is a half-turn rotation group.
This verifies a conjecture made by Klee, Nevo, Novik and Zhang for the case when $\Gamma$ is a point-inversion group. Our result implies that
Stanley's lower bound theorem for centrally symmetric polytopes extends to pseudomanifolds with a free simplicial involution, thus verifying (the inequality part) of another conjecture of  Klee, Nevo, Novik and Zheng.
Both results actually apply to a much larger class of simplicial complexes, namely the circuits of the simplicial matroid.  The proof of our rigidity result adapts earlier ideas of Fogelsanger to the setting of symmetric simplicial complexes.

\medskip \noindent {\bf Keywords:} infinitesimal rigidity, pseudomanifold, symmetric simplicial complex, lower bound theorem

\end{abstract}

\section{Introduction}
Let $\cS$ be a pure $k$-dimensional abstract simplicial complex with vertex set $V(\cS)$ and edge set $E(\cS)$.
The lower bound theorem concerns the invariant $g_2(\cS)$ defined by
\begin{equation}
    \label{eqn_g2}
    g_2(\cS) =  |E(\cS)| -(k+1)|V(\cS)| + \binom{k+2}2. 
\end{equation}
Barnette~\cite{B73} showed that 
$g_2(\cS) \geq 0$ if $\cS$ is the boundary complex of a $(k+1)$-dimensional
convex polytope.
This was later generalised to simplicial spheres (Stanley \cite{stanley_g_theorem}) and pseudomanifolds 
(Kalai \cite{kalai}, Tay \cite{tay}). 

Readers familiar with the rigidity theory of bar-joint frameworks will note that the
right hand side of (\ref{eqn_g2}) arises naturally in that theory and indeed is nonnegative
if the 1-skeleton of $\cS$ is a generically rigid graph in $\mathbb R^{k+1}$.
This connection between rigidity theory and polytopal combinatorics, noted by Kalai \cite{kalai} and 
Gromov \cite{Gromov}, has been fundamental to much of the 
work on lower bound theorems for various classes of simplicial complexes. 

Around the same time as \cite{kalai}, Fogelsanger proved that the 1-skeleton of 
a minimal homology $k$-cycle is generically rigid in $\mathbb R^{k+1}$ for $k\geq 2$ \cite{fog}. It is not 
difficult to see that a $k$-pseudomanifold is a minimal homology $k$-cycle in the sense of Fogelsanger. Thus Fogelsanger's result provides a generalisation and independent proof of lower bound theorem 
for pseudomanifolds.  Fogelsanger's proof technique is 
different to the previous  proofs of the lower bound theorem: it is a direct proof based on rigidity lemmas for edge contraction, vertex splitting 
and an ingenious decomposition result (more on this below). 

In the present paper we consider $\mathbb{Z}_2$-symmetric simplicial complexes i.e. simplicial complexes $\cS$ with  a free  
automorphism of order two. Such complexes were referred to as centrally symmetric simplicial complexes in \cite{KNNZ}. The term $\mathbb{Z}_2$-symmetric is more suitable for our purposes since we will consider realisations of complexes in $\R^{k+1}$ which have an arbitrary symmetry of order two, not just a point inversion through the origin.

Note that if $P$ is a centrally symmetric simplicial polytope in $\mathbb R^{k+1}$, that is,  $-P = P$, then the boundary complex of $P$ is a $\mathbb{Z}_2$-symmetric simplicial complex of dimension $k$.
Stanley \cite{Stanleycs} showed that if $\cS$ is the boundary complex of a centrally symmetric simplicial $(k+1)$-polytope for some $k \geq 3$, then 
\begin{equation}
    \label{eqn_cs_lb}
    g_2(\cS) \geq \binom{k+1}2 - (k+1)
\end{equation}
Later Sanyal, Werner and Ziegler \cite{SWZ} used the rigidity-based approach to obtain further properties of the $f$-vector of centrally symmetric polytopes. More recently Klee, Nevo, Novik and Zheng \cite{KNNZ} 
used techniques from rigidity theory to characterise the cases when equality can hold in Stanley's Theorem.
While centrally symmetric simplicial polytopes have been the topic of much research since Stanley's paper, the extension of this theory  from simplicial polytopes to simplicial complexes is less developed than in the non-symmetric case.
One significant result in that direction is the recent
proof by Novik and Zheng of the $\mathbb Z_2$-symmetric upper bound conjecture for simplicial spheres \cite{novik2020highly}.
In this paper we address the lower bound conjecture for a larger class of $\mathbb Z_2$-symmetric simplicial complexes, namely the circuits of the simplicial
matroid. Precise definitions will be given in Section \ref{sec_bk}, for now it suffices to note that a simplicial $k$-circuit is a minimal homology $k$-cycle over $\mathbb Z_2$ in Fogelsanger's terminology. 

Our approach is to extend Fogelsanger's rigidity theorem and proof 
techniques to the $\mathbb{Z}_2$-symmetric setting. 
The extension is not straightforward and, in particular, makes use of a new notion of frameworks with partial symmetry which we have not seen in the rigidity literature. 
Our extension of Fogelsanger's decomposition technique is powerful enough to show that the graph of any $\mathbb{Z}_2$-symmetric simplicial $k$-circuit can be realised as an infinitesimally rigid framework having a specified point group symmetry of order two in $\mathbb{R}^{k+1}$, unless $d=2$ and the point group is the half-turn rotation group in $\mathbb{R}^3$. 

We shall see that when $k=2$ and the point group is the half-turn rotation group in $\mathbb{R}^3$, every symmetric realisation of a $\mathbb{Z}_2$-symmetric planar graph is infinitesimally flexible.
The smallest such example is the famous Bricard octahedron~\cite{bricard1897memoire}, given in  Figure~\ref{fig:bricard}. 
The Bricard octahedron was the source of inspiration for Connelly's flexible polytope.
More generally, the rigidity  of  symmetric and non-convex realisations of 1-skeletons of polytopes has been one of the central topics in rigidity theory, see, e.g.,~\cite{servatius18}.
We believe that our rigidity theorem will have a substantial impact in rigidity theory because it gives the first 
extension of rigidity results on symmetric convex polytopes to a more general family including non-convex symmetric realisations of simplicial complexes.

\begin{figure}
\centering
\includegraphics[scale=0.6]{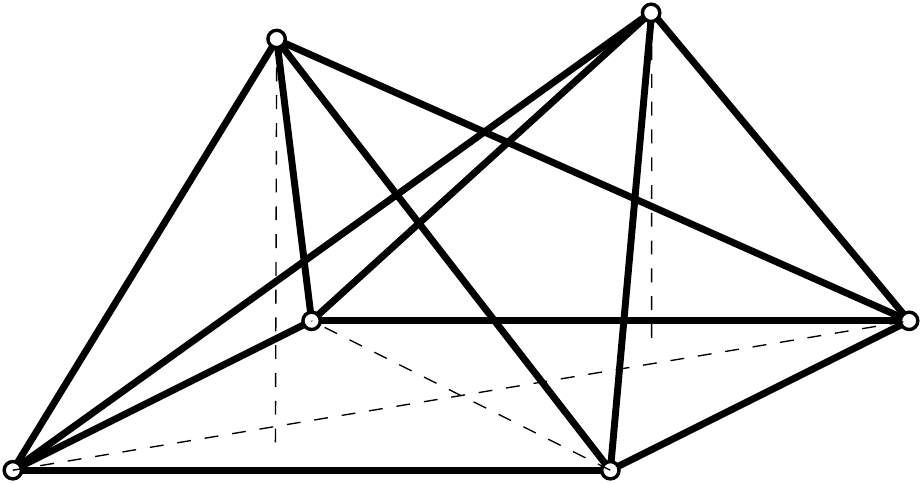}
\caption{The Bricard octahedron.}
\label{fig:bricard}
\end{figure}

The paper is organised as follows.
We first give preliminary facts on the infinitesimal rigidity of frameworks under a point group symmetry in Section 
\ref{sec_pre_rigidity}.
In Section~\ref{sec_bk}, we review Fogelsanger's decomposition 
technique for simplicial circuits using
an approach from \cite{CJT1}.
We then provide an extension of Fogelsanger's rigidity theorem 
to realizations of simplicial $k$-circuits in which several pairs of vertices are constrained to lie  symmetrically in $\R^{k+1}$.
In Section~\ref{sec_cs_comp}, we prove our main theorem on the rigidity of $\mathbb{Z}_2$-symmetric simplicial circuits.
The application to the lower bound theorem is given in Section~\ref{sec_lbt}.
We close the paper by giving some final remarks and open problems in Section~\ref{sec_conclusion}.

\section{Preliminaries on Rigidity and Symmetry}\label{sec_pre_rigidity}
In this section we introduce some basic  results on the infinitesimal rigidity of frameworks, and their extensions to the symmetric case.

Throughout the paper we use the following basic notation.
For a graph $G$, the vertex set  and the edge  set are denoted by $V(G)$ and $E(G)$, respectively.
For $X\subseteq V(G)$, let $E_G[X]=\{uv\in E(G): u, v\in X\}$.
The subgraph of $G$ {\em induced} by $X$ is given by $G[X]=(X, E_G[X])$.
For $v\in V(G)$, $N_G(v)$ denotes the set of vertices adjacent to $v$ in $G$.
\subsection{Infinitesimal rigidity}

A graph drawn in Euclidean space with straight line edges is called a {\em (bar-joint) framework}
and denoted by a pair $(G,p)$ of the graph $G$ and the point configuration $p: V(G)\rightarrow \mathbb{R}^d$. We will also refer to  $(G,p)$ as a {\em realisation} of $G$ in $\R^d$.

An {\em infinitesimal motion} of a framework $(G,p)$ is a map $\dot{p}: V(G)\rightarrow \mathbb{R}^d$ satisfying 
\begin{equation}\label{eq:inf}
(p(i)-p(j))\cdot (\dot{p}(i)-\dot{p}(j))=0\qquad (ij\in E(G)).
\end{equation}
For a skew-symmetric matrix $S$ and $t\in \mathbb{R}^d$, 
it is not difficult to check that $\dot{p}$ defined by $\dot{p}(i)=Sp(i)+t\ (i\in V(G))$ is an infinitesimal motion of $(G,p)$.
Such an infinitesimal motion is called {\em trivial}.
The framework $(G,p)$ is {\em infinitesimally rigid} if every infinitesimal motion of $(G,p)$ is trivial.
We say that the graph $G$ is {\em rigid} in $\mathbb{R}^d$ if $G$ has an  infinitesimally rigid realisation in $\mathbb{R}^d$.

Since (\ref{eq:inf}) is a system of linear equations in $\dot{p}$, we can represent it by a matrix of size $|E(G)|\times d|V(G)|$.
This matrix is called the {\em rigidity matrix} $R(G,p)$ of $(G,p)$.
When the affine span of $p(V(G))$ has dimension at least $d-1$, the dimension of the space of trivial motions is $\binom{d+1}{2}$, so
$(G,p)$ is infinitesimally rigid if and only if $\rank R(G,p)=d|V(G)|-\binom{d+1}{2}$. 
In particular, if $(G,p)$ is infinitesimally rigid, then 
\begin{equation}\label{eq:rigidity_bound}
|E(G)|-d|V(G)|+\binom{d+1}{2}\geq 0.
\end{equation}
Kalai \cite{kalai} and Gromov \cite{Gromov} exploited the close relationship between (\ref{eqn_g2}) and  (\ref{eq:rigidity_bound})  to use rigidity theory to prove their extensions of the Lower Bound Theorem.
We will use a similar approach to obtain our lower bound theorem for $\Z_2$-symmetric simplicial complexes.

\subsection{Rigidity under symmetry}
A {\em graph with a vertex pairing} is a pair $(G,\ast)$ of a graph $G$ and a free  involution $\ast$ acting on some  $X\subseteq V(G)$. (Thus $\ast: X\rightarrow X$ and satisfies  $\ast(\ast(u))=u$ and $\ast(u)\neq u$ for all $u\in X$.)
We will denote $\ast(u)$ by $u^*$ for each $u\in X$ and put $Y^*=\{u^*: u\in Y\}$ for all $Y\subseteq X$. In addition, for each  
$W \subset V(G)$ we put
\[ X_W= (X \cap W) \cap (X \cap W)^*.\] Then $\ast$ induces a free involution $\ast_W: X_W \rightarrow X_W$.
We will simply denote $\ast_W$ by $\ast$ when it is clear from the context.
Similarly, for a subgraph $H$ of $G$, $(H,\ast_{V(H)})$ is simply denoted by $(H,\ast)$.

We will use the free involution $\ast$ to force a  symmetry on the point configuration  in a possible realisation of $G$.
Let $\Gamma$ be a point group of $\mathbb{R}^d$ of order two, that is, a subgroup of $O(d)$ of order two.
A {\em $\Gamma$-framework} is a triple $(G,\ast,p)$ of a graph $G$, a free involution $\ast: X\rightarrow X$,
and a point-configuration $p$ such that $p(u^*)=\gamma (p(u))$ for all $u\in X$,
where $\gamma$ is the non-identity element in $\Gamma$.
By ignoring $\ast$, $(G,\ast,p)$ can be considered as  a realisation of $G$, 
and hence we can apply the above terminology for frameworks to $(G,\ast,p)$.
We will sometimes refer to a  $\Gamma$-framework $(G,\ast)$ as a {\em $\Gamma$-symmetric realisation} of $(G,\ast)$.
We say that $(G,\ast)$ is {\em $\Gamma$-rigid} if $(G,\ast)$ has an infinitesimally rigid $\Gamma$-symmetric realisation.

A $\Gamma$-framework $(G,\ast,p)$ is {\em generic} if the transcendence
degree of the set of coordinates of all the points in $p(V)$ over $\mathbb Q$ takes the maximum possible value  $d(|V| - |X|/2)$. Thus $(G,\ast)$ is $\Gamma$-rigid if and only if every (or equivalently, some) generic $\Gamma$-symmetric realisation of $(G,\ast)$ is infinitesimally rigid. 

\begin{example} This example illustrates that the
$\Gamma$-rigidity of a fixed $(G,\ast)$ may depend on the group $\Gamma$.
Consider the case when $d=3$. Then there are three different types of point groups of order two corresponding to point inversion, rotation, and reflection. Suppose that $G$ is the graph of the octahedron and $\ast$ maps each vertex to its antipodal vertex.
If $\Gamma$ is a rotation group, then any $\Gamma$-framework $(G,\ast,p)$ 
is the 1-skeleton of a Bricard octahedron, which is 
infinitesimally flexible, see Example \ref{ex:2} below. On the other hand, we will show that
every generic $\Gamma$-framework $(G,\ast,p)$ is infinitesimally rigid when $\Gamma$ is  a point inversion or  reflection group. 
\end{example}

We will concentrate our attention on graphs $G$ with a vertex pairing $*:X \rightarrow X$ 
with the property that  
$xx^* \not\in E$ for all $x\in X$. We will refer to such a vertex pairing as a {\em non-adjacent vertex pairing}.
Note that, if $\ast: X \rightarrow X$ is a non-adjacent vertex pairing of $G$,
then the induced bijection  $*_W: X_W \rightarrow X_W$
 is again a non-adjacent vertex pairing of $G[W]$ for all $W\subseteq V(G)$.

We emphasise that a non-adjacent vertex pairing $*:X\rightarrow X$ need not,
in general, induce an involution on the edge set $E[X]$. 
Our main concern, however, is the special case  when $X=V(G)$ and $\ast$ is an automorphism of $G$ without fixed edges.
In this case $(G,\ast)$ is said to be {\em a $\mathbb{Z}_2$-symmetric graph}.
Since $\ast$ is an automorphism, $\ast(e)$ is well defined and we abbreviate $\ast(e)$ by $e^*$ for each edge $e$. 
We consider the more general class of graphs with a non-adjacent vertex pairing because it arises naturally in our analysis of the $\Z_2$-symmetric case. 

We next derive a stronger inequality than (\ref{eq:rigidity_bound}) for 
the number of edges in a $\Gamma$-rigid realisation of a 
$\mathbb{Z}_2$-symmetric graph.
For integers $d \geq 1$ and $0 \leq t \leq d$ let $I_{t,d}$ be the
the diagonal matrix of size $d$ whose first $t$ diagonal entries are 1 and remaining $d-t$ diagonal entries are $-1$.
Let $\Gamma_{t,d}$ be the point group generated by $I_{t,d}$.
Observe that if $\Gamma$ is a point group of order 2 in $\mathbb R^d$
then we can always choose a coordinate system so that $\Gamma = \Gamma_{t,d}$
for some $0\leq t \leq d-1$. 

\begin{lemma}\label{lem:symmetric_maxwell}
Suppose that $0 \leq t \leq d-1$ and let $\Gamma=\Gamma_{t,d}$.
If a $\mathbb{Z}_2$-symmetric graph $(G,\ast)$ is $\Gamma$-rigid and $|V(G)|\geq 2d$, then
\begin{equation}\label{eq:symmetric_maxwell}
|E(G)|\geq d|V(G)|-2\min\left\{\binom{t+1}{2}+\binom{d-t}{2}, \binom{d+1}{2} - \binom{t+1}{2}-\binom{d-t}{2}\right\}.
\end{equation}
\end{lemma}
\begin{proof}
Let $n=|V(G)|$ and $m=|E(G)|$. 
Since $(G,\ast)$ is $\Gamma$-rigid, there is a $\Gamma$-symmetric infinitesimally rigid realisation $(G, \ast,p)$ of $(G,\ast)$. Since $|V(G)|\geq 2d$, we can take such a realisation so that the affine span of $p(V(G))$ is at least $d-1$. 
In particular, the space of trivial motions of $(G,*,p)$ has dimension $\binom{d+1}{2}$.

Recall that the rigidity matrix $R(G,p)$ represents a linear map from $\mathbb{R}^{dn}$ to $\mathbb{R}^m$.
We shall decompose $\mathbb{R}^{dn}$ and $\mathbb{R}^m$ into two subspaces whose elements are either symmetric or anti-symmetric with respect to $\Gamma$ and $*$.
Specifically, let 
\begin{align*}
M_{\rm sym}&:=\{ \dot{p}:V(G)\rightarrow \mathbb{R}^d \mid \dot{p}(u^*)=I_{t,d} \dot{p}(u)\mbox{ for all } u\in V(G)\} \\
M_{\rm ant}&:=\{ \dot{p}:V(G)\rightarrow \mathbb{R}^d \mid \dot{p}(u^*)=-I_{t,d} \dot{p}(u)\mbox{ for all } u\in V(G)\} \\
S_{\rm sym}&:=\{ w:E(G)\rightarrow \mathbb{R} \mid w(e^*)=w(e) \mbox{ for all } e\in E(G)\} \\
S_{\rm ant}&:=\{ w:E(G)\rightarrow \mathbb{R} \mid w(e^*)=-w(e) \mbox{ for all } e\in E(G)\}. 
\end{align*}
Then $\mathbb{R}^{dn}=M_{\rm sym}\oplus M_{\rm ant}$ and $\mathbb{R}^m=S_{\rm sym}\oplus S_{\rm ant}$.
Observe further that the rigidity matrix maps  $M_{\rm sym}$  to $S_{\rm sym}$ and $M_{\rm ant}$  to $S_{\rm ant}$.
Indeed, if $\dot{p}\in M_{\rm sym}$, then for any edge $e=uv$,
we have 
$(p(u^*)-p(v^*))\cdot (\dot{p}(u^*)-\dot{p}(v^*))= (I_{t,d}(p(u^*)-p(v^*)))\cdot I_{t,d}((\dot{p}(u^*)-\dot{p}(v^*))=(p(u)-p(v))\cdot (\dot{p}(u)-\dot{p}(v))$, which implies that the image of $\dot{p}$ belongs to $S_{\rm sym}$.
A similar calculation shows the corresponding property for $M_{\rm ant}$.

Let $T$ be the space of trivial infinitesimal motions of $(G,p)$. 
A canonical basis of $T$ consists of $\binom{d}{2}$ infinitesimal rotations about the subspaces spanned by each set of $(d-2)$ axes and $d$ translations along each axis.
Due to the structure of $I_{t,d}$, it follows that the $\binom{t+1}{2}$-dimensional space of isometries in the subspace spanned by the first $t$ axes and the $\binom{d-t}{2}$-dimensional space of rotations rotating in the subspace spanned by the last $d-t$ axes are contained in $M_{\rm sym}$. 
One can also directly check that the remaining   $\binom{d+1}{2} - \binom{t+1}{2}-\binom{d-t}{2}$ elements of the canonical basis belong to $M_{\rm ant}$. Hence 
\begin{equation*}
\dim T\cap M_{\rm sym}= \binom{t+1}{2}+\binom{d-t}{2} \text{ and }
\dim T\cap M_{\rm ant} =  \binom{d+1}{2} - \binom{t+1}{2}-\binom{d-t}{2}.
\end{equation*}

The infinitesimal rigidity of $(G,p)$ implies that $\ker R(G,p)=T$.
Since the rigidity matrix maps  $M_{\rm sym}$  to $S_{\rm sym}$ and $M_{\rm ant}$  to $S_{\rm ant}$, 
this gives
\begin{align*}
\frac{m}{2}&=\dim S_{\rm sym}\geq \dim M_{\rm sym}-\dim T\cap M_{\rm sym}=\frac{dn}{2}- \left(\binom{t+1}{2}+\binom{d-t}{2}\right) \\
\frac{m}{2}&=\dim S_{\rm ant}\geq \dim M_{\rm ant}-\dim T\cap M_{\rm ant}=\frac{dn}{2}- \left( \binom{d+1}{2} - \binom{t+1}{2}-\binom{d-t}{2}\right)
\end{align*}
as required.
\end{proof}

\begin{example}\label{ex:2}
Suppose $(G,*)$ is $\Gamma$-rigid in $\R^3$. Then,  Lemma~\ref{lem:symmetric_maxwell}
implies 
\[
|E(G)|\geq 3|V(G)|-\begin{cases}
6 & \text{ if $t=0$ i.e. $\Gamma$ is generated by a point inversion, } \\
4 & \text{ if $t=1$ i.e. $\Gamma$ is generated by a half turn rotation, } \\
6 & \text{ if $t=2$ i.e. $\Gamma$ is generated by a reflection}.
\end{cases}
\]
If $G$ is the graph of the octahedron, $|E(G)|=12$ and $|V(G)|=6$.
So it cannot be $\Gamma$-rigid if $\Gamma$ is generated by a half turn rotation, which is the case of the Bricard octahedron.
\end{example}

Note that the case when $d=3$ is exceptional since
the right side of (\ref{eq:symmetric_maxwell}) is maximised at $t=1$ when $d=3$, and at $t=0$ when $d\geq 4$.

The infinitesimal rigidity of frameworks having a point group symmetry is an extensively studied topic in rigidity theory.
See, for example, \cite{ST,SW} for symmetric extensions of classical rigidity theorems.

\subsection{Gluing properties}

We write $\aff(Y)$ for the affine span of a subset $Y$ of $\mathbb R^d$ and $\lin(Y)$ for its linear span.
The gluing properties of rigid frameworks and graphs are important ingredients  in the proof of Fogelsanger's Rigidity Theorem. For completeness, and since we cannot find an authoritative source 
for these results in the literature, we include details of these 
gluing properties 
in the classical non-symmetric setting.

\begin{theorem}
    \label{thm_gluing_frmwks}
    Let $(G,p)$ be a framework in $\mathbb R^d$.
    Suppose that $G_1, G_2$
    are subgraphs of $G$ such that $(G_i,p|_{V(G_i)})$ is infinitesimally rigid for  $i=1,2$ and, $\aff(p(V(G_1) \cap V(G_2)))$ has dimension at least $d-1$. Then $(G,p)$ is infinitesimally rigid. 
\end{theorem}
\begin{proof}
Suppose that $\dot{p}$ is an infinitesimal flex of $(G,p)$.
By assumption there exist skew symmetric matrices $A_i$ of size $d$ and $t_i \in \mathbb R^d$ for $ i =1,2$ such that $\dot{p}(v) = A_ip(v) +t_i$
for $v \in V(G_i)$. We can choose coordinates so that $p(w) = 0$ for 
some $w \in V(G_1) \cap V(G_2)$. It follows that $A_1 0 +t_1 = 
\dot{p}(w) = A_2 0 + t_2$. Therefore $t_1 = t_2$ and so $A_1p(v) =
A_2p(v)$ for all $v \in V(G_1) \cap V(G_2)$. Therefore the skew symmetric matrix $A_1 - A_2$ vanishes on a space of dimension $d-1$ and so
must be $0$ as required.
\end{proof}

\begin{corollary}
    \label{cor_gluing_graphs}
    Suppose that $G_1,G_2$ are rigid graphs in $\mathbb R^d$. If $|V(G_1) \cap V(G_1)| \geq d$ then $G_1 \cup G_2$
    is rigid in $\mathbb R^d$.
    \qed
\end{corollary}

Now we extend Corollary \ref{cor_gluing_graphs} to $\Gamma$-rigidity.
First we prove a lemma about the affine span of $\Gamma$-symmetric subsets of $\mathbb R^d$.

\begin{lemma}
\label{lem_affine}
    Let $\gamma = I_{t,d}$ for $0 \leq t \leq d-1$.
    Suppose that $P$ is a generic set of points in 
    $\mathbb R^d$ and $|P| = n$. Then $\dim(\aff(P \cup \gamma(P))) =\min\{n,d-t\}
    +\min\{n-1,t\}$.
\end{lemma}

\begin{proof}
    Let $V$, respectively $W$, be the eigenspace of $\gamma$ corresponding to the eigenvalue $1$, respectively $-1$ and let 
    $\pi_V$, respectively $\pi_W$, be the orthogonal projection from 
    $\mathbb R^d$ onto $V$, respectively $W$. Since $P$ is generic in 
    $\mathbb R^d$ and $\pi_V, \pi_W$ are projections onto coordinate subspaces, it follows that $\pi_V(P)$ is generic in $V$ and
    $\pi_W(P)$ is generic in $W$. 
    Therefore 
    $\dim(\aff(\pi_V(P))) = \min\{n-1,t\}$ and 
    $\dim(\lin(\pi_W(P))) = \min\{n,d-t\}$. 
    \begin{claim}
        $\aff(P \cup \gamma(P)) = \aff(\pi_V(P)) + \lin(\pi_W(P))$.
    \end{claim}
    \begin{proof}[Proof of claim.]
        Suppose $x = \sum_{q \in P \cup \gamma(P)} c_q q \in \aff(P \cup \gamma(P)))$ for scalars $c_q$ with $\sum_q c_q=1$. 
        Now  $x = \pi_V(x) + \pi_W(x)$ and $\pi_V(x) = \sum_{p \in P}
        (c_p + c_{\gamma(p)})\pi_V(p)$ and $\pi_W(x) = \sum_{p \in P}
        (c_p - c_{\gamma(p)})\pi_W(p)$. Since $\sum_{p \in P}
        (c_p+c_{\gamma(p)}) = \sum_{q \in P\cup \gamma(P)} c_q = 1$, 
        it follows that $\aff(P \cup \gamma(P)) \subset \aff(\pi_V(P)) + \lin(\pi_W(P))$. 
        
        On the other hand, consider $\sum_{p \in P} d_p \pi_V(p) \in \aff(\pi_V(P))$
        and $ \sum_{p \in P} b_p \pi_W(p) \in \lin(\pi_W(P)) $.
        Set $c_p=\frac{d_p+b_p}{2}$ and $c_{\gamma(p)}=\frac{d_p-b_p}{2}$ for each $p\in P$.
        Then $\sum_{q \in P \cup \gamma(P)} c_q = \sum_{p \in P} d_p = 1$.
        It follows that $\aff(\pi_V(P)) + \lin(\pi_W(P)) \subset \aff(P \cup \gamma(P))$. 
    \end{proof}
The facts that
$\aff(\pi_V(P)) \subseteq V$, $\lin(\pi_W(P))\subseteq W$ and $V\cap W=\{\bf 0\}$ (since $V,W$ are eigenspaces for distinct eigenvalues of $\gamma$) now  give
\begin{align*}
\dim(\aff(P \cup \gamma(P))) =
\dim (\aff(\pi_V(P)) + \lin(\pi_W(P)))&=\dim\aff(\pi_V(P)) + \dim\lin(\pi_W(P))\\
&=
\min\{n,d-t\}
    +\min\{n-1,t\},
\end{align*}
as required.
\end{proof}

\begin{theorem}[Gluing Theorem]
    \label{thm_pcs_gluing}
    Let $\Gamma = \Gamma_{t,d}$ for $0 \leq t \leq d-1$. Let $(G,\ast)$ be a graph with a  vertex pairing $\ast: X\rightarrow X$. Suppose  that $H_1=(V_1,E_1)$ and $H_2=(V_2,E_2)$ are  subgraphs of $G$ whose union is $G$ and that, for $i=1,2$, 
    $(H_i,\ast)$ is $\Gamma$-rigid in $\mathbb R^d$.
    Let $m = |X_{V_1} \cap X_{V_2}|$.
    \begin{enumerate}
        \item[(a)] If $m >0$ and  $$|V_1 \cap V_2| \geq d+m-1-\min\{m/2,d-t\}-\min\{m/2-1,t\}$$ then $(G,\ast)$ is $\Gamma$-rigid.
        \item[(b)] If $m =0$ and   $$|V_1 \cap V_2| \geq d$$ then $(G,\ast)$ is $\Gamma$-rigid. 
    \end{enumerate}
\end{theorem}

\begin{proof}
    Let $(G,\ast,p)$ be a generic $\Gamma$-framework. Then, for $i=1,2$,
    $(H_i, \ast, p|_{V_i})$ is generic and therefore infinitesimally rigid. 
    Now we observe that, by Lemma \ref{lem_affine}, the dimension of 
    $\aff(p(V_1 \cap V_2))$ is 
    $\min\{|V_1\cap V_2| - m +\min\{m/2,d-t\} +\min\{m/2-1,t\},d\}$
    if $m>0$, and is
    $\min\{|V_1\cap V_2| -1,d\}$ if $m=0$.
    The theorem now follows by applying Theorem \ref{thm_gluing_frmwks}
    to the frameworks $(H_i,p|_{V_i})$, $i=1,2$.
\end{proof}

\subsection{Vertex splitting}

Whiteley's Vertex Splitting Lemma \cite[Proposition 1]{Whiteley_splitting} is a fundamental result in rigidity theory and plays a key role in the proof of Fogelsanger's Rigidity Theorem.
We will derive a version of this lemma for $\Gamma$-rigidity.

First we fix some terminology.
Let $G$ be a graph, $u \in V(G)$, $C \subset N_G(u)$
and $D\subset N_G(u) \setminus C$. Let $G'$ 
be the graph obtained from $G$ by deleting 
the edges $uw$ for all $w \in D$, adding a new vertex $u'$ 
and adding edges $u'z, z \in C \cup D \cup \{u\}$. 
We say that $G'$ is {\em obtained from 
$G$ by splitting $u'$ from $u$ along $C$}, or more succinctly, 
by vertex splitting at $u$. 

\begin{theorem}[Whiteley's Vertex Splitting~\cite{Whiteley_splitting}]
    Let $(G,p)$ be a framework in $\mathbb R^d$ such that $R(G,p)$ is row independent. 
    Also let $u \in V(G)$ and $C \subset N_G(u)$ such that 
    $|C| \leq d-1$ and $\{p(w) - p(u): w \in C\}$ is linearly independent. Let $G'$ be the graph obtained
    from $G$ by splitting $u'$ from $u$ along $C$. Then there is some $z \in \mathbb R^d$ such that 
    $R(G',q)$ is row-independent, where $q(w) = p(w)$ for all $w \in V(G)$ and $q(u') = p(u)+z$. \qed
    \label{thm_whiteley}
\end{theorem}

We will derive two versions of Theorem \ref{thm_whiteley} for preserving the $\Gamma$-rigidity of a graph $G$ with a
non-adjacent vertex pairing $*:X\to X$.
The first, Lemma \ref{lem_whiteley_pcs_contraction2}, will be applied when the split vertex $u$ does not belong to $X$. The second, Theorem \ref{thm_whiteley_pcs_contraction}, simultaneously splits two vertices  $x,x^*\in X$.

First a remark on the inverse operation to vertex splitting: edge contraction. Given a graph $G$ and an edge $uv\in E(G)$, we use  $G/uv$ to denote the simple graph obtained from $G$ by contracting  $v$ onto $u$. More precisely $G/uv = (G - v) \cup \{uz: z \in N_G(v)\}$.
Note that, if $*:X \rightarrow X$ is a non-adjacent vertex pairing and either
$X \cap \{u,v\} = \emptyset$, or $u\in X$, $v \not\in X$  and $u^*v \not\in E(G)$, then $*$ is also a
non-adjacent vertex pairing on $G/uv$.

\begin{lemma}
    Suppose $\Gamma$ is a point group of $\mathbb R^d$ of order two.
    Let $(G,\ast)$ be a graph with a non-adjacent vertex pairing $*:X
    \rightarrow X$ and $uv \in E(G)$ with $u,v \not\in X$. Suppose that there is $C \subset N_G(u) \cap N_G(v)$ such that $|C| = d-1$ and $|X_C| \leq 2$. If $(G/uv,*)$ is $\Gamma$-rigid in $\mathbb R^d$ then $(G,*)$ is $\Gamma$-rigid in $\mathbb R^d$.
    \label{lem_whiteley_pcs_contraction2}
\end{lemma}

\begin{proof}
    Observe that $X_{C \cup \{u\}} = X_C$ since $u \not\in X$.
    If $|X_C| = 0$ then for any $\Gamma$-generic framework $(G/uv,\ast,p)$ in $\mathbb R^d$,
    $p(C\cup \{u\})$ is a generic set of $d$ points in $\mathbb R^d$. 
    If $|X_C| = 2$ then $p(C \cup \{u\})$ is an affinely independent 
    set of $d$-points in $\mathbb R^d$. 
    In both cases, it follows that 
    $\{p(u) - p(z): z\in C\}$ is linearly independent and so we can use Theorem \ref{thm_whiteley} to 
    construct a $\Gamma$-framework $(G,\ast,q)$ that is infinitesimally rigid  in $\mathbb R^d$.
\end{proof}

We next give a variant of Theorem \ref{thm_whiteley} which deals with a symmetric splitting of two distinct vertices. 
Its proof 
is essentially the same as Whiteley's proof of Theorem \ref{thm_whiteley} so we only give the details which are different.

\begin{theorem}
    Let $\tau:\mathbb R^d \rightarrow \mathbb R^d$ be a non-singular linear transformation.  
    Suppose that $(G,p)$ is a framework in $\mathbb R^d$ such that $R(G,p)$ is row-independent.
    Also let $u,v \in V(G)$  be non-adjacent vertices, 
    $C_1 \subset N_G(u), C_2 \subset N_G(v)$ such that $|C_1|,|C_2| \leq d-1$, and the sets  $\{p(w)-p(u):w \in C_1\}$  and $\{p(x) - p(v): x \in C_2\}$ are both linearly independent. Let $G'$ be the graph 
    obtained by splitting $u$ from $u'$ along $C_1$ and then splitting $v$ from $v'$ along $C_2$. Then there is some 
    $z \in \mathbb R^d$ such that $R(G',q)$ is row-independent, where $q(w) = p(w)$
    for all $w \in V(G)$, $q(u') = p(u) +z$ and $q(v') = p(v) + \tau(z)$.
    \label{thm_2v_split}
\end{theorem}

\begin{proof}[Proof sketch.]
    Choose $ y \in \mathbb R^d$ that is not an element of any of the 
    spaces 
    $\lin\{p(u) - p(w): w \in C_1\}$, $\lin\{p(v) - p(x): x \in C_2\}$,
    $\tau^{-1}(\lin\{p(u) - p(w): w \in C_1\})$ or $\tau^{-1}(\lin\{p(v) - p(x): x \in C_2\})$. This 
    is possible since each of these linear spaces 
    has dimension at most $d-1$. For $t \in \mathbb R$, let $q_t: V(G')
    \rightarrow \mathbb R^d$ be defined by 
    $q_t(w) = p(w), w \in V(G)$, $q_t(u') = p(u)+ty$, 
    $q_t(v') = p(v) +\tau(ty)$. Now the argument from
    the proof of 
    \cite[Proposition 1]{Whiteley_splitting} can be used to 
    show that,
    for sufficiently small non-zero $t$, $R(G',q_t)$ 
    is row-independent.
\end{proof}

We next apply Theorem \ref{thm_2v_split} to the setting of 
$\Gamma$-rigidity. 
Suppose that $G=(V,E)$ has 
a non-adjacent vertex pairing $*:X \rightarrow X$ and $x, y \in X$ with $xy,x^*y^*\in E$. 
If $x^*y, xy^* \not\in E$ then the restriction 
of $*$ to $X\sm \{y,y^*\}$
is a non-adjacent vertex pairing of $(G/xy)/x^*y^*$. We will abuse notation and continue to  
use $*$ for this vertex pairing of $(G/xy)/x^*y^*$.

\begin{theorem}\label{thm_whiteley_pcs_contraction}
Let $\Gamma$ be a point group of order two in $\mathbb{R}^d$.
    Let $(G,\ast)$ be a graph with a non-adjacent vertex pairing $\ast:X\rightarrow X$, $x,y \in X$ such that $xy, x^*y^* \in E(G)$ and $xy^*, x^*y \not\in E(G)$. Suppose that
    there exist $C \subset N_G(x)\cap N_G(y)$ and $D \subset N_G(x^*) \cap N_G(y^*)$ such that $|C|, |D| = d-1$ 
    and $|X_C|, |X_D| \leq 2$. Let $G' = (G/xy)/x^*y^*$. If $(G',*)$
    is $\Gamma$-rigid, then $(G,*)$ is $\Gamma$-rigid.  
\end{theorem}
\begin{proof}
    Let $(G',\ast,p)$ be a generic $\Gamma$-framework. 
    Then $(G',p)$ is infinitesimally rigid by assumption. 
    Let $I = \{xv: v\in C\} \cup \{x^*w: w \in D\}$. 
    Observe that $X_{C \cup \{x\}} = X_C$ and $X_{D \cup \{x^*\}} = X_D$. 
    Since $(G',\ast,p)$ is a generic $\Gamma$-framework and $|X_C|, |X_D| \leq 2$, it follows that both $\{p(w) - p(x): w \in C\}$ 
    and $\{p(z) - p(x^*): z \in D\}$ are  linearly independent. This also implies  that the set of rows of $R(G',p)$ labelled by $I$ is linearly independent. 
    Choose a maximal independent row-set that contains the row-set labeled by $I$ and let $J$ be the corresponding set of edges of $G$. 
    Since $(G',p)$ is infinitesimally rigid, it follows that $J$ spans $V(G')$ and $|J| = d|V(G')| - \binom{d+1}2$.
    Let $G'[J]=(V(G'),J)$. 
    We apply Theorem \ref{thm_2v_split} (with $\tau$ being the non-identity element of $\Gamma$)
    to the framework $(G'[J],p)$ to obtain a $\Gamma$-framework $(G'',\ast,q)$, where 
    $G''$ is obtained from $G'[J]$ by splitting  $x$ from $y$ along $C$ and then splitting $x^*$ from $y^*$ along $D$. 
    By Theorem \ref{thm_2v_split}, $(G'',q)$ is infinitesimally rigid.
    Since $G''$ is a spanning subgraph of $G$, 
    $(G,\ast, q)$ is the required infinitesimally rigid $\Gamma$-symmetric realisation of $(G,\ast)$.
\end{proof}

We emphasise again that, for an arbitrary graph $G = (V,E)$ with a non-adjacent vertex pairing $*:X \rightarrow X$, the hypothesis that $xy \in E[X]$ does not imply that 
$x^*y^* \in E$. Hence, in order to apply Theorem \ref{thm_whiteley_pcs_contraction} to $G$, we must check that $xy,x^*y^*\in E$  and  similarly that $xy^*, x^*y\not\in E$.

\section{Background on Simplicial Complexes}
\label{sec_bk}

We now consider simplicial complexes.
We summarise some notation and results from \cite{CJT1} that we will
need later. We refer the reader to \cite{CJT1} for more details on  this material.

Our main results will apply to a certain class of pure abstract simplicial $k$-complexes. However, as in \cite{CJT1},
it will be convenient for us to consider a larger family of complexes in which multiple copies of the 
same facet may exist. Thus we define a {\em simplicial $k$-multicomplex} to be a multiset whose elements 
are $(k+1)$-sets. Suppose that $\cS$ is a simplicial $k$-multicomplex. 
For $0\leq j\leq k$, a $j$-face of $\cS$ is a $(j+1)$-set $F$ that is a subset of some 
element of $\cS$. 
In the case that $\cS$ does not contain multiple copies of any $k$-face
we say that $\cS$ is a simplicial $k$-complex. We justify this terminology by noting that a pure abstract simplicial complex 
of dimension $k$, in the usual sense of that term, corresponds to 
a unique simplicial $k$-complex in our sense.
For $k \geq 1$ the graph of $\cS$, denoted $G(\cS)$, is the 
simple graph whose vertex set, $V(\cS)$, is
the set of $0$-faces of $\cS$ and whose edge set, $E(\cS)$, is the set of $1$-faces of $\cS$.

The {\em boundary} of a simplicial $k$ multicomplex $\cS$ is given by $$\partial \cS = \{ F \subset V(\cS): |F| = k, F \text{ is contained in an odd number of 
elements of \cS }\}.$$ By definition $\partial \cS$ is a simplicial $(k-1)$-complex. We say that 
$\cS$ is a {\em simplicial $k$-cycle} if $\partial \cS = \emptyset$ and that $\cS$ is a {\em simplicial $k$-circuit} if 
$\cS$ is a non-empty simplicial $k$-cycle and no proper subset of $\cS$ is a simplicial $k$-cycle.
A {\em trivial simplicial $k$-circuit} is a simplicial $k$-multicomplex comprising of two copies of the same 
$k$-face. Our use of the word `circuit' here comes from matroid theory, since the 
set of simplicial
$k$-circuits contained within any simplicial $k$-multicomplex $\cS$ forms the set of circuits of a matroid on $\cS$. In 
particular the simplicial 1-circuits in a multigraph are the circuits of its graphic matroid. 
Note that if $\cS$ is a simplicial $k$-multicomplex and $\cU \subset \cS$ is a simplicial $k$-cycle then $\partial (\cS\setminus \cU)
= \partial \cS$. This implies that a non-empty simplicial $k$-cycle can be partitioned into a disjoint union of 
simplicial $k$-circuits. 
We will frequently  consider the symmetric difference $\cS \triangle \cT$ of two simplicial $k$-complexes $\cS, \cT$ and use the fact that $\partial (\cS \triangle \cT) = (\partial \cS) \triangle (\partial \cT)$.\footnote{We will only use this operation for simplicial $k$-complexes so do not need to give a definition for the symmetric difference of two multisets.}

We say that a simplicial $k$-multicomplex $\cS$ is {\em strongly 
connected} if, for any distinct $U,W \in \cS$, there is a sequence 
$U = U_1,\dots,U_k = W$ in $\cS$, such that $|U_i \cap U_{i+1}| = k$ for $i = 1,\dots,k-1$. Observe that, if $\cT$ is a maximal strongly connected subset of $\cS$, then $\partial T \subset \partial \cS$. In particular it follows easily from this observation that any simplicial $k$-circuit is strongly connected. 

We next define a contraction operation for two vertices $u,v$ in a simplicial $k$-multicomplex $\cS$. 
Let $\ant(\{u,v\}) = \{U \in \cS: \{u,v\} \not\subset U\}$ be the {\em anti-star} of $\{u,v\}$ in $\cS$.
Then $\cS/uv$ is the simplicial $k$-multicomplex obtained from $\ant(\{u,v\})$ by 
replacing every $k$-face $U$ that contains
$v$ with $U-v+u$. We say that $\cS/uv$ is obtained from $\cS$ by  {\em contracting $v$ onto $u$}. 
Let $\gamma :\ant(\{u,v\}) \rightarrow \cS/uv$ be the canonical bijection. 
Note that our contraction operation 
may create multiple copies of a $k$-face in $\cS/uv$ even when $\cS$ is a simplicial complex. We allow this in order to have the useful property that the set of simplicial $k$-cycles is closed under the contraction operation. Note also that, if every  edge in $E(\cS) \setminus \{u,v\}$ belongs to at least one $k$-face in $\ant(\{u,v\})$, then
$G(\cS/uv) = G(\cS)/uv$, where the right hand side denotes the usual contraction operation on simple graphs.

We next describe Fogelsanger decomposition of simplicial $k$-circuits. Suppose that 
$\cS$ is a nontrivial simplicial $k$-circuit and $uv \in E(\cS)$. Then $\cS/uv$ is a simplicial $k$-cycle so we can express it as 
$\cS/uv = \cS'_1 \sqcup \dots \sqcup \cS'_m$, where $\cS'_j$ is a simplicial $k$-circuit for $1\leq j \leq m$
(this partition is not necessarily unique). Let 
\begin{align*}
    \cS_j &= \gamma^{-1}(\cS'_j), \\
    \cS_j^\dagger &= \{K \subset V(\cS): \{u,v\} \subset K, K-u, K-v \in \partial \cS_j\}, \\
    \cS_j^+ &= \cS_j \cup \cS_j^\dagger.
\end{align*}
We say that $(\cS_1^+,\dots,\cS_m^+)$ is a {Fogelsanger decomposition for $\cS$ at $uv$}. The
properties of this decomposition are summarised in the following lemma which is a restatement of 
\cite[Lemma 3.9]{CJT1}.

\begin{lemma}\label{lem_fog_props}
Suppose  $\cS$ is a nontrivial simplicial $k$-circuit and  $uv\in E(\cS)$. Let $(\cS_1^+,\cS_2^+, \dots, \cS_m^+)$  be a Fogelsanger decomposition of $\cS$ at $uv$.
Then: 
\begin{enumerate}[(a)]
    \item \label{en1:a} $\cS_i^+/uv$ is a simplicial $k$-circuit  for all $1\leq i\leq m$;
    \item \label{en1:b} $\cS_i^+$ is a nontrivial simplicial $k$-circuit for all $1\leq i\leq m$ and each $K\in \cS_i^+\sm\cS$ is a clique of $G(\cS)$ which contains $\{u,v\}$;
    \item \label{en1:c} each $k$-face of $\ant(\{u,v\})$ is a $k$-face in a unique $\cS_i^+$;
    \item \label{en1:d} $\cS = \triangle_{j=1}^m \cS_{j}^+$;
    \item \label{en1:e} $uv\in E(\cS_i^+)$  for all $1\leq i\leq m$ and $\bigcup_{i=1}^m E(\cS_i^+)=E(\cS)$;
    \item \label{en1:f} for all proper $ I\subset \{1,2,\ldots,m\}$, there exists $j\in \{1,2,\ldots,m\}\sm I$  and a $(k+1)$-clique $K$ of $G(\cS)$ such that $K \not\in \cS$ and $K\in (\triangle_{i\in I}\cS_i ^+)\cap \cS_{j}^+$.
\end{enumerate}
\end{lemma}

\section{$\Gamma$-rigidity of Simplicial Circuits with a Non-adjacent Vertex Pairing}

\label{sec_fams}

In this section we will consider simplicial circuits with a non-adjacent vertex pairing and use the results of Section \ref{sec_pre_rigidity} to prove that their graphs are 
$\Gamma$-rigid in certain cases.
We begin with an analysis of the graph of a crosspolytope in this context.
These graphs will serve as the base case in  the inductive proof of our main theorem.

For $k \geq 0$ let $e_1,e_2,\ldots, e_{k+1}$ be the standard basis for $\R^{k+1}$.
The {\em $(k+1)$-dimensional crosspolytope} is the convex hull of the set of points $\{\pm e_1,\pm e_2,\ldots, \pm e_{k+1}\}$. 
We will use $\cB_{k}$ to denote the boundary complex of this polytope. It is well known that $\cB_{k}$ is a simplicial $k$-complex whose vertex set
is $\{\pm e_i: 1\leq i \leq k+1\}$. Moreover, the $k$-faces  are 
precisely the transversals of $\{ \{\pm e_1\}, \{\pm e_2\}, \dots,\{\pm e_{k+1}\}\}$.
In particular for any distinct vertices  $u,v$ of $\cB_{k}$, $uv$ is an edge $\cB_{k}$ if and only if $v \neq -u$. 
Hence, there is the unique non-adjacent vertex pairing  $\ast:V(\cB_k)\rightarrow V(\cB_k)$ that pairs antipodal vertices of $\cB_k$. 
Using this unique $\ast$, the graph $G(\cB_k)$ of $\cB_k$ is $\mathbb{Z}_2$-symmetric. 

We now check the $\Gamma_{t,k+1}$-rigidity of $(G(\cB_{k}), \ast)$. 
We need the following operation and result due to Whiteley. 
Given a graph $G=(V,E)$, the {\em cone} of $G$ is the graph $G^v$ obtained by adding a new vertex $v$ and all edges from $v$ to $V$. 

\begin{lemma}[Coning~\cite{whiteley1983cones}]
\label{lem:cone} Suppose that $G=(V,E)$ is a graph, $G^v$ is the cone of $G$ and $p$ is a realisation of $G^v$ in $\R^{k+1}$
such that $p(V)$ is contained in a hyperplane $H$, $p(v)\not\in H$ and $(G,p|_V)$ is infinitesimally rigid in $H$ (viewed as a copy of $\R^{k}$). Then $(G^v,p)$ is   infinitesimally rigid in $\R^{k+1}$.
\end{lemma}

\begin{lemma}\label{lem:cross} Suppose $k \geq 2$ and $\Gamma=\Gamma_{t,k+1}$ for some $0 \leq t \leq k$.
Let $G(\cB_{k})$ be the graph of the $(k+1)$-dimensional crosspolytope and $\ast:V(\cB_{k})\rightarrow V(\cB_{k})$  be the non-adjacent vertex pairing on $V(\cB_{k})$.
Then $(G(\cB_{k}),\ast)$ is $\Gamma$-rigid in $\mathbb{R}^{k+1}$  unless $k=2$ and $\Gamma$ is a rotation group of order two.
\end{lemma}
\begin{proof}
Denote the set of vertices of $\cB_{k}$ by $\{x_1, x_1^*, \dots, x_{k+1}, x_{k+1}^*\}$.
We show by induction on $k$ that $G(\cB_k)$ has an infinitesimally  rigid $\Gamma_{t,k+1}$-symmetric realisation $p_{t,k}$ in $\R^{k+1}$. For the base case when $k=2$, it is straightforward to check that: $(G(\cB_{2}),p_{0,2})$  is infinitesimally rigid when $p_{0,2}(x_i)=e_i$  and $p_{0,2}(x_i^*)=-e_i$ for all $1\leq i\leq 3$; $(G(\cB_{2}),p_{2,2})$  is infinitesimally rigid when $p_{2,2}(x_i)=e_i+e_3$  and $p_{2,2}(x_i^*)=e_i-e_3$ for all $1\leq i\leq 3$.  Hence we may assume that $k\geq 3$.  

The inductive step will follow immediately from:

\begin{claim}
Suppose $(G(\cB_{k-1}),\ast)$ has an  infinitesimally rigid $\Gamma_{t,k}$-symmetric realisation $p_{t,k-1}$  in $\R^k$  for some $k\geq 3$ and $0\leq t\leq k-1$. 
Then $(G(\cB_{k}),\ast)$ has both an infinitesimally rigid $\Gamma_{t,k+1}$-symmetric realisation  and an infinitesimally rigid $\Gamma_{t+1,k+1}$-symmetric realisation in $\mathbb R^{k+1}$.
\end{claim}
\begin{proof}   
We may assume that $p_{t,k-1}$  is a  generic $\Gamma_{t,k}$-symmetric realisation of $(G(\cB_{k-1}),\ast)$  in $\R^k$ and that $V(\cB_{k-1}) = 
\{x_1,x_1^*,\dots,x_k,x_k^*\}$.

We first extend $p_{t,k-1}$ to an infinitesimally rigid $\Gamma_{t,k+1}$-symmetric realisation $p_{t,k}$ of  $(G(\cB_{k}),\ast)$ in $\mathbb R^{k+1}$ by putting 
$p_{t,k}(z) = (p_{t,k-1}(z),0)$ for all $ z\in V(\cB_{k-1})$ and 
$p_{t,k}(x_{k+1})=e_{k+1}=-p_{t,k}(x^*_{k+1})$.  
Then the restrictions of $p_{t,k}$ to both $G(\cB_{k})-x_{k+1}$ and $G(\cB_{k})-x_{k+1}^*$ are infinitesimally rigid by Lemma \ref{lem:cone}. We can now use Theorem \ref{thm_pcs_gluing} to deduce that $p_{t,k}$ is an infinitesimally rigid realisation of $G(\cB_{k})$.

A similar proof works for $p_{t+1,k}$. We construct  
a realisation  $p_{t+1,k}$ of $G(\cB_{k})$ from $p_{t,k-1}$ by putting
$p_{t+1,k}(z) = (0,p_{t,k-1}(z))$ for all $ z\in V(\cB_{k-1})$ and 
$p_{t+1,k}(x_{k+1})=e_{1}=p_{t+1,k}(x^*_{k+1})$.   Then we can use  Lemma \ref{lem:cone} and Theorem \ref{thm_pcs_gluing} to deduce that  
$(G(\cB_{d}),p_{t+1,k})$ is infinitesimally rigid. 
\end{proof}
\end{proof}

We next use Lemma \ref{lem:cross} to analyze the case when $X\neq V(\cB_k)$. 

\begin{lemma}\label{lem:cross1}
   Let $k\geq 2$, $X\subseteq V(\cB_k)$, $\ast:X\to X$ be a non-adjacent vertex pairing of $G(\cB_k)$ and $\Gamma$ be a point group of $\R^{k+1}$ of order two. Suppose that $|X|\leq 2k$. 
Then $(G,\ast)$ is $\Gamma$-rigid in $\mathbb R^{k+1}$.
\end{lemma}
\begin{proof}
It will suffice to show that $(G(\cB_k),\ast)$ has an infinitesimally rigid  $\Gamma$-symmetric realisation in $\R^{k+1}$. This follows immediately from Lemma \ref{lem:cross} unless $k=2$ and $\Gamma$ is a rotation group. In the latter case, a realisation of $G(\cB_k)$  as the regular octahedron is also a $\Gamma$-symmetric realisation of $(G(\cB_k),\ast)$ since 
$|X|\leq 4$. Hence the $\Gamma$-rigidity of $(G(\cB_2),\ast)$ follows from the infinitesimal rigidity of the 1-skeleton of the regular octahedron.
\end{proof}

We next show in that Lemma \ref{lem:cross1} can be extended from $\cB_k$ to arbitrary simplicial $k$-circuits.
This result will be a key ingredient in the proof of our main theorem (Theorem \ref{thm_real_main}). 

\begin{theorem}
    \label{thm_smallX}
    Let $G=(V,E)$ be the graph of a simplicial $k$-circuit $\cS$ for some $k\geq 2$, $X\subseteq V$, $*:X\to X$ be a non-adjacent vertex pairing of $G$, and $\Gamma$ be a point group of $\R^{k+1}$ of order two. Suppose that $|X|\leq 2k$. 
Then $(G,\ast)$ is $\Gamma$-rigid in $\mathbb R^{k+1}$.
\end{theorem}
\begin{proof}
    Suppose, for a contradiction, that $\cS$ is a counterexample with as few vertices as possible. Clearly, $\cS$ cannot be a trivial simplicial 
    $k$-circuit. 
Let 
$(G,\ast,p)$ be a generic $\Gamma$-framework in $\mathbb R^{k+1}$. 
First we record a useful observation. 
Suppose that 
$K \subset V$ is a clique in $G$. Then
    since $xx^* \not\in E$ for all $x \in X$ (as $\ast$ is non-adjacent), it follows that 
    \begin{equation}
        \text{$X_K = \emptyset$ and $p(K)$ is a generic set of points in $\mathbb R^{k+1}$.}
        \label{eqn_clique1}
    \end{equation}

    \begin{claim}
        Every edge of $G$ is incident to $X$.
        \label{clm_edgesX1}
    \end{claim}

    \begin{proof}[Proof of claim.]
        Suppose, for a contradiction, that $uv \in E$ and $u,v \not\in X$. Let $(\cS_1^+, \dots,\cS_m^+)$
        be a Fogelsanger decomposition of $\cS$ with respect to $uv$. Put $(V_i, E_i) = G_i = G(\cS_i^+)$ 
        and $X_i = 
        X_{V_i} = (V_i \cap X) \cap (V_i \cap X)^*$. Then $|X_i| \leq |X|$ for each $i$. 
        Also $\ast|_{X_i}$ is a non-adjacent vertex pairing for $G_i$ by Lemma~\ref{lem_fog_props}(e).

Suppose $|V_i| < |V|$ for all $i = 1,\dots,m$.
        Then, by the minimality of  $|V|$, $(G_i,\ast|_{X_i})$ is $\Gamma$-rigid in $\mathbb R^{k+1}$.
It follows that 
        $(G_i,p|_{V_i})$ is infinitesimally rigid. 
        Now, a straightforward induction argument using Lemma \ref{lem_fog_props}(\ref{en1:f}), Theorem \ref{thm_pcs_gluing}(b) and
        (\ref{eqn_clique1}) 
proves that $(G,p)$ is infinitesimally rigid, 
        contradicting our choice of $\cS$.

        Thus we can assume without loss of generality that $V_1 = V$. Since $\cS_1^+/uv$ is a simplicial $k$-circuit it follows 
        from the minimality of  $|V|$ that $(G_1/uv,\ast)$ 
        is $\Gamma$-rigid.
        By Lemma~\ref{lem_fog_props}(b), we can choose $U \in \cS_1^+$ such that $u,v \in U$.
        Since $\cS_1^+$ is a nontrivial simplicial circuit, 
        there is some $w \not\in U$ such that 
        $w \in N_{G_1}(u) \cap N_{G_1}(v)$. Now $C = U -\{u,v\}+w$ satisfies, $|C| = k$,
        and $|X_C| \leq 2$. Lemma \ref{lem_whiteley_pcs_contraction2} now implies that 
        $(G_1,\ast)$ is $\Gamma$-rigid in $\mathbb R^{k+1}$ and, since $G_1$ is a spanning 
        subgraph of $G$, that $(G,\ast)$ is also 
        $\Gamma$-rigid in $\mathbb R^{k+1}$. This contradicts our choice of $\cS$.
    \end{proof}

For each $U\in \cS$, Claim \ref{clm_edgesX1} gives $|U\cap X|\geq k$. Since $|X|\leq 2k$ and $xx^*\not\in E(\cS)$ for all $x\in X$ we have 
\begin{equation}\label{eqn_2k}
\mbox{$|X|=2k$ and $|U\cap X|=k$ for all $U\in \cS$.}
\end{equation}

\begin{claim}\label{clm:cross} Some subcomplex of $\cS$ is isomorphic to $\cB_k$.
\end{claim}
\begin{proof} 
Choose $U\in \cS$. Since $xx^*\not\in E(\cS)$ for all $x\in X$, (\ref{eqn_2k}) implies that we can label the elements of $X$ as $x_1,x_1^*,x_2,x_2^*,\ldots, x_k,x_k^*$ with $U\cap X=\{x_1,x_2,\ldots,x_k\}$. Let $U=\{x_1,x_2,\ldots,x_k,u\}$. Since $\cS$ is a nontrivial simplicial $k$-circuit, (\ref{eqn_2k}) also implies that $U-x_i+x_i^*$ is a $k$-face of $\cS$ for all $1\leq i\leq k$ and hence that $J+u$ is a $k$-face of $\cS$ for all transversals $J$ of $\{\{x_1,x_1^*\},\{x_2,x_2^*\},\ldots, \{x_k,x_k^*\}\}$. 

These cannot be all of the $k$-faces of $\cS$ since they do not form a simplicial 
$k$-circuit. Therefore there is some $u' \in V(\cS)\setminus X$ such that 
$u' \neq u$ and the same argument as before shows that 
$J+u'$ is a $k$-face of $\cS$ for all transversals $J$ of $\{\{x_1,x_1^*\},\{x_2,x_2^*\},\ldots, \{x_k,x_k^*\}\}$.
Hence every transversal of $\{\{x_1,x_1^*\},\{x_2,x_2^*\},\ldots, \{x_k,x_k^*\},\{u,u'\}\}$ is a $k$-face of $\cS$. These transversals induce a subcomplex of $\cS$ which is isomorphic to  $\cB_k$.
\end{proof}  

Since both $\cS$ and $\cB_k$ are simplicial $k$-circuits, Claim \ref{clm:cross} implies that $\cS\cong \cB_k$. We can now use Lemma \ref{lem:cross1} to deduce that $(G,\ast)$ is $\Gamma$-rigid.
\end{proof}

It is worth noting that the case $X = \emptyset$ in Theorem \ref{thm_smallX} is Fogelsanger's theorem for simplicial $k$-circuits, so we can view Theorem \ref{thm_smallX} as a non-generic
extension of this fundamental result in rigidity theory.

The Bricard octahedron shows that the hypothesis $|X| \leq 2k$ in  Theorem \ref{thm_smallX} is necessary when $k=2$ and $\Gamma$ is a half turn rotation group.

\section{$\Z_2$-Symmetric Simplicial Complexes}
\label{sec_cs_comp}

Suppose that $\cS$ is a simplicial $k$-multicomplex. Let $*:V(\cS)
\rightarrow V(\cS)$ be an involution. For $X\subseteq V(\cS)$ and 
$\cU \subset \cS$ let $X^*=\{x^*:x\in X\}$ and $\cU^* = \{K^*: K \in \cU\}$. The set $X$, respectively  $\cU$, is $*$-invariant if $X^*=X$, respectively  $\cU^*=\cU$. We say that $*$ is a {\em simplicial involution} if, for every facet $F$ of $\cS$,  $F^*$ is a facet of $\cS$ of the same multiplicity as $F$, and that  $*$ is {\em free} if $V(F^*) \neq V(F)$ for every face $F$ of $\cS$. A {\em $\mathbb{Z}_2$-symmetric simplicial $k$-multicomplex} is a pair $(\cS,\ast)$ where $\cS$ is a simplicial $k$-multicomplex and $*$ is a free simplicial involution on $\cS$. This terminology is consistent with our terminology for graphs since, if $(\cS,*)$ is a $\mathbb{Z}_2$-symmetric $k$-multicomplex, then  $(G(\cS),\ast)$ is a $\mathbb{Z}_2$-symmetric graph.
We will often abuse notation by omitting explicit mention of the involution $\ast$ when it is obvious from the context.

\subsection{$\mathbb{Z}_2$-irreducible cycles and their structural properties}
We will refer to a $\mathbb{Z}_2$-symmetric $k$-multicomplex  which is also a simplicial $k$-cycle as a {\em $\mathbb{Z}_2$-symmetric $k$-cycle}.  A {\em \csir{$k$}} is a  
non-empty $\mathbb{Z}_2$-symmetric $k$-cycle $(\cS,*)$ which is minimal in the sense that, for all   $\emptyset\neq \cT \subsetneq \cS$, $(\cT,*|_{V(\cT)})$  is not a $\mathbb{Z}_2$-symmetric $k$-cycle. We say that $\cS$ is a {\em trivial \csir{$k$}} if it consists of two vertex disjoint copies of the trivial simplicial $k$-circuit and the involution interchanges the vertex sets of these two trivial simplicial $k$-circuits. Otherwise $\cS$ is a {\em nontrivial \csir{$k$}}.

Suppose that $\cS$ is a non-empty $\mathbb{Z}_2$-symmetric $k$-cycle. 
It is straightforward to show that if $\cT\subseteq \cS$ is a $\mathbb{Z}_2$-symmetric $k$-cycle then $\cS\sm \cT$ 
is a $\mathbb{Z}_2$-symmetric $k$-cycle.  This implies that there exists a partition $\cS = \cS_1 \sqcup \dots \sqcup \cS_m$ where $\cS_i$ is a \csir{$k$} for $i = 1,\dots,m$. This partition is not necessarily unique (even up to permutations).

Clearly, if $\cS$ is a simplicial $k$-circuit with a free simplicial involution then it is also a \csir{$k$}. However there are 
\csir{$k$}s that are not simplicial $k$-circuits - the trivial \csir{$k$} is an example. See also Figure~\ref{fig_cs_circuits} for a nontrivial example. Our first result characterises \csir{$k$}s which are not simplicial $k$-circuits.

\begin{lemma}
    Let $(\cS,*)$ be a nontrivial \csir{$k$}. Then either $\cS$ is a simplicial $k$-circuit or 
    $\cS = \cT \sqcup \cT^*$ where $\cT$ is a nontrivial simplicial $k$-circuit. 
    Moreover, if the second alternative holds and $\cU$ 
    is a simplicial $k$-circuit properly contained in $\cS$, then $\cU = \cT$ or $\cU = \cT^*$.
    \label{lem_cs_circuits}
\end{lemma}

\begin{proof}
    Suppose that $\cS$ is not a simplicial $k$-circuit. Then some  simplicial $k$-circuit $\cT$ is properly contained in $\cS$. 
    Since $\cS$ is a \csir{$k$}, $\cT^* \neq \cT$. 
    It follows that $\cT \triangle \cT^*$ is a non-empty $*$-invariant 
    simplicial $k$-cycle contained in $\cS$, so $\cT \triangle \cT^* = \cS$, whence $\cS 
    = \cT \sqcup \cT^*$. If $\cT$ is a trivial simplicial $k$-circuit then 
    $K = V(\cT) \cap V(\cT^*)$ is a face of $\cS$ and $K^* = K$. 
    Since $*$ is a free involution it follows that $K = \emptyset$
    and so $\cS$ is a trivial \csir{$k$}, contradicting the hypothesis. Hence $\cT$ is nontrivial.

    Now suppose that the second alternative in the lemma holds and $\cU$ 
    is a simplicial $k$-circuit properly contained in $\cS$.  Then $\cS = \cU \sqcup \cU^*$ by the same argument as the first paragraph and
    $\cU \triangle \cT = (\cU^* \cap \cT) \cup (\cU \cap \cT^*)$ is a 
    $*$-invariant simplicial $k$-cycle contained in $\cS$. Therefore,
    either $\cU \triangle \cT = \emptyset$ and so $\cU = \cT$ or, 
    $\cU\triangle \cT = \cS$ and so $\cU = \cS \setminus \cT = \cT^*$.
\end{proof}

An example illustrating the second alternative in Lemma \ref{lem_cs_circuits} is given in Figure \ref{fig_cs_circuits}.

\begin{figure}[ht]
\begin{center}
\includegraphics[scale=0.8]{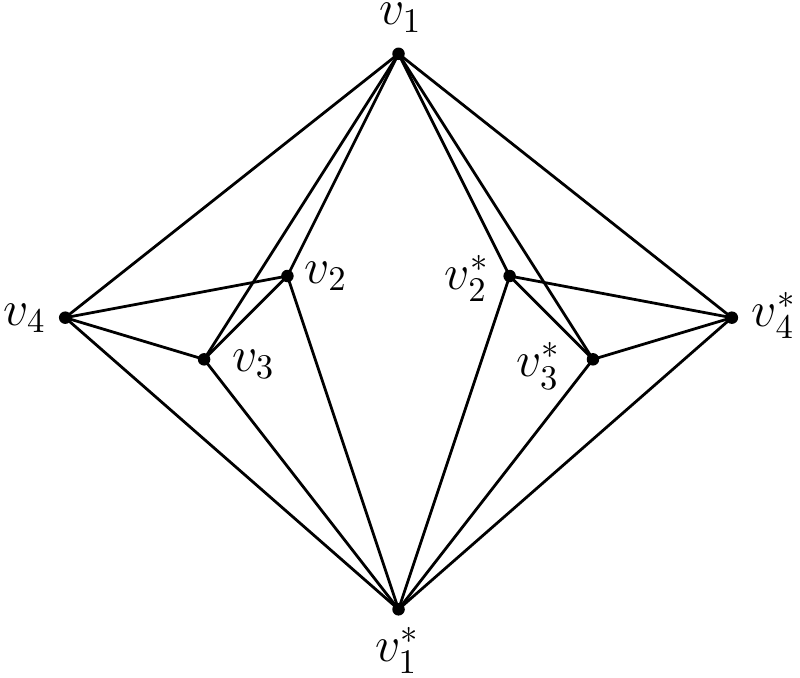}
\end{center}
\caption{The 1-skeleton of a \csir{$2$} $\cS$, which is not a simplicial 2-circuit. We have $\cS=\cT\sqcup \cT^*$ where 
$\cT=\{\{v_1,v_2,v_3\},\{v_1,v_2,v_4\},\{v_1,v_3,v_4\},\{v_1^*,v_2,v_3\},\{v_1^*,v_2,v_4\},\{v_1^*,v_3,v_4\}\}$
is the simplicial 2-circuit given by the boundary complex of the hexahedron.}
\label{fig_cs_circuits}
\end{figure}

\begin{lemma}
    Suppose that $(\cS,*)$ is a \csir{$k$} for some $k \geq 1$. Then $|V(\cS)| \geq 2k+2$ with equality if and only if 
    $(\cS,*)$ is a trivial \csir{$k$} or $\cS$ is isomorphic to $\cB_k$.
    \label{lem_small_cs_circuits}
\end{lemma}

\begin{proof}
We have   $|V(\cS)| \geq 2k+2$ since, for any $U \in \cS$ we have $U \cap U^* = \emptyset$ and hence $|V(\cS)| \geq |U \sqcup U^*|=2k+2$.
    
    Suppose that $|V(\cS)| = 2k+2$. Let $V(\cS)=\{x_i,x_i^*:1\leq i\leq k+1\}$. If $\cS$ contains two copies of some $k$-face $W$ then it follows easily that $\cS = \{\{ W , W, W^*, W^*\}\}$ and so $\cS$ is a trivial \csir{$k$}. Thus we may assume that each $W\in \cS$ has multiplicity one. Then, for each $W \in \cS$ and $w \in W$, we have $W-w+w^* \in \cS$ since $W-w$ is contained in at least two $k$-faces of $\cS$ and $*$ is free. It follows that every transversal 
    of $\{ \{x_i,x_i^*\}: 1\leq i\leq k+1\}$ is a $k$-face of $\cS$ and so $\cS \cong \cB_k$.
\end{proof}

\begin{lemma}
    Let $(\cS,*)$ be a \csir{$k$} for some $k\geq 2$ and $X\subset V(\cS)$ be an inclusion-minimal $\ast$-invariant vertex separator of $G(\cS)$. Suppose $|X|\leq 2k$. Then either $\cS=\cT\sqcup \cT^*$ for some simplicial  $k$-circuit $\cT$ with $V(T)\cap V(T^*)=X$, or $|X|=2k$ and the subgraph of $G(\cS)$ induced by $X$ is isomorphic to $G(\cB_{k-1})$.
    \label{lem_smallsep}
    \label{lem_sep_2k}
\end{lemma}
\begin{proof}
    Choose $V_1, V_2 \subset V(\cS)$ such that $V_1 \cap V_2 = X$, $G(\cS) = 
    G(\cS)[V_1] \cup G(\cS)[V_2]$. Let $\cT_i 
    = \{W \in \cS: W \subset V_i\}$. Since $X$ is a vertex separator
    of $G(\cS)$ it follows that $\cS = \cT_1 \cup \cT_2$ and that $\cT_1, \cT_2 \neq \cS$. 
    The hypothesis that $X$ is $*$-invariant and the fact that $|X| \leq 2k$
imply that, for any $W\in \cS$, $W\not \subseteq X$; hence $\cT_1 \cap \cT_2 = \emptyset$.
    Thus 
    $\cS = \cT_1 \sqcup \cT_2$, $\partial \cT_1 = \partial \cT_2=:\cU$ and $V(\cU) \subset X$. 

Suppose   $\cU = \emptyset$.  Then $\cT_i$ is a simplicial $k$-cycle for $i=1,2$. Lemma \ref{lem_cs_circuits} now implies that  $\cT_1$ is a simplicial $k$-circuit, $\cT_2=\cT_1^*$  and $\cS=\cT_1\sqcup \cT_1^*$. The minimality of $X$ now gives $V(\cT_1)\cap V(\cT_1^*)=X$. Hence we may assume that $\cU \neq \emptyset$.

The definition of the boundary operator implies that $\cU$  is a simplicial $(k-1)$-cycle and each of its facets has multiplicity one. 
Hence $\cU \triangle \cU^*$ is a $\mathbb{Z}_2$-symmetric $(k-1)$-cycle that does not contain a trivial simplicial $(k-1)$-circuit. Since $V(U)\subseteq X$ and $X$ is $*$-invariant, $\cV:=\cU \triangle \cU^*\subseteq X$. We can now apply  
Lemma \ref{lem_small_cs_circuits} and the hypothesis that $|X| = 2k$ to $\cV$ to deduce that either $\cV = \emptyset$ or $\cV \cong \cB_{k-1}$.

Suppose $\cV \cong \cB_{k-1}$. Then $|X|=2k$ and $\cV$ is the \csir{$(k-1)$}
    consisting of all transversals of the following partition of $X$:
    $\{ \{x_1,x_1^*\},\dots,\{x_k,x_k^*\}\}$. 
Hence every $k$-subset of $X$ which contains no pair $\{x_i,x_i^*\}$ is contained in $\cV$. 
Since $*$ is free and $V(\cU)\subset X$, this implies that $\cU \subset \cV$. Since $U$ is a simplicial $(k-1)$-circuit,  
    $\cV = \cU \sqcup \cU^*$ and hence $\cV$ is a trivial \csir{$(k-1)$}. This contradicts the assumption that $\cV \cong \cB_{k-1}$.

Hence
    $\cV=\cU \triangle \cU^* = \emptyset$ so $\cU=\cU^*$. Then $\cU$ is a $\mathbb{Z}_2$-symmetric $(k-1)$-cycle with at most $2k$ vertices that 
    does not contain a trivial \csir{$(k-1)$}. Lemma \ref{lem_small_cs_circuits} now implies that 
    $\cU \cong \cB_{k-1}$ and hence $G(\cS)[X]$ is isomorphic to $G(\cB_{k-1})$.
\end{proof}

\subsection{$\mathbb{Z}_2$-symmetric Fogelsanger decomposition}
We next adapt Fogelsanger's decomposition technique for simplicial $k$-circuits to the context of \csir{$k$}s. Let $(\cS,*)$ be  
a \csir{$k$} and put $G = (V,E)=G(\cS)$. Suppose that 
$xy \in E$ and $xy^* \not\in E$. Then $x^*y^* \in E$, $x^*y \not\in E$ and $(\cS/xy)/x^*y^*$ is 
a \csir{$k$} under the free simplicial involution induced by $*$ on $V-y-y^*$.
Recall that, for a face $F$ of $\cS$, the antistar of $F$ in $\cS$ is given by  $\ant(F) = \{ U \in \cS: F \not\subset U\}$. 
Observe that there is a bijection $\gamma: \ant(\{x,y\}) \cap \ant(\{x^*,y^*\}) \rightarrow
(\cS/xy)/x^*y^*$ given by $\gamma (U) = U$ if $y,y^* \not\in U$, $\gamma(U) = U -y+x$ if
$ y \in U$, and $\gamma(U)  = U -y^*+x^*$ if $ y^* \in U$.

Choose a partition $\{\cS_1',\cS_2',\ldots,\cS_m'\}$
of  $(\cS/xy)/x^*y^*$ into \csir{$k$}s. 
For $1\leq i \leq m$, let $\cS_i = \gamma^{-1}(\cS'_i)$, 
$$\cS^\dagger_i = \{K \subset V(\cS_i):\, \{x,y\} \subset K, K - x\in \partial \cS_i,\, K-y \in \partial \cS_i\}$$
and
\begin{equation}
    \label{eqn_fog_part}
    \cS_i^+ = \cS_i \cup \cS^\dagger_i \cup (\cS^\dagger_i)^*.
\end{equation}
We say that 
$(\cS_1^+, \dots, \cS_m^+)$ is a {\em $\mathbb{Z}_2$-symmetric Fogelsanger decomposition} for $\cS$ at $xy$.
Note that a $\mathbb{Z}_2$-symmetric Fogelsanger decomposition at $xy$ is defined only when $xy^*\notin E$.

The \csir{$2$} in Figure \ref{fig_cs_circuits} can be used to illustrate this construction.
We have $(\cS/v_2v_3)/v_2^*v_3^*=\cS_1'\sqcup \cS_2'$ where $\cS_1',\cS_2'$ are trivial \csir{$2$}s on $\{v_1,v_2,v_3,v_1^*,v_2^*,v_3^*\}$ and  
$\{v_1,v_2,v_4,v_1^*,v_2^*,v_4^*\}$, respectively. The above definitions now give 
$S_1=\cT_1\cup \cT_1^*$, where $\cT_1=\{\{v_1,v_2,v_3\},\{v_1,v_2,v_4\},\{v_1,v_3,v_4\}\}$, $\cS_1^\dagger=\{\{v_1,v_2,v_3\}\}$ and $\cS_1^+$ is the disjoint union of the boundary complexes of two tetrahedra on $\{v_1,v_2,v_3,v_4\}$ and $\{v_1^*,v_2^*,v_3^*,v_4^*\}$, respectively. Similarly,   $\cS_2^+$ is the disjoint union of the boundary complexes of two tetrahedra on $\{v_1^*,v_2,v_3,v_4\}$ and $\{v_1,v_2^*,v_3^*,v_4^*\}$, respectively. 

The properties of $\mathbb{Z}_2$-symmetric  Fogelsanger decompositions given in the following lemma
are analogous to those for Fogelsanger decompositions given in Lemma \ref{lem_fog_props}.

\begin{lemma}
    Let $(\cS,*)$ be a \csir{$k$} for $k \geq 2$ and $xy\in E(\cS)$  with $xy^*\not\in E(\cS)$. Suppose that  $(\cS_1^+, \dots, \cS_m^+)$ is a $\mathbb{Z}_2$-symmetric  Fogelsanger decomposition for $\cS$ at $xy$. Let $G(\cS_i^+)=(V_i,E_i)$ for $1 \leq i \leq m$. Then:
\begin{enumerate}[(a)]
    \item\label{cs-a} $xy, x^*y^* \in E_i$ for $1 \leq i \leq m$;
    \item\label{cs-b} $\bigcup_{i=1}^m E_i = E$ and $\bigcup_{i=1}^m V_i = V$;
    \item\label{cs-c}  $(\cS_i^+/xy)/x^*y^*$ is a \csir{$k$} for $1 \leq i \leq m$;
    \item\label{cs-d} $\cS_i^+$ is a nontrivial \csir{$k$} and each $K \in \cS_i^+\setminus \cS$ is a clique of $G$ that either contains $\{x,y\}$ or contains $\{x^*,y^*\}$;
    \item\label{cs-e}  if $U\in\cS$ and $U$ contains neither $\{x,y\}$ nor $\{x^*,y^*\}$ then $U$ belongs to a unique $\cS_i^+$;
    \item\label{cs-f}  $\cS = \triangle_{k=1}^m \cS_k^+$;
    \item\label{cs-g}  for all proper subsets $I$ of $\{1,\dots,m\}$, there exists $j \in \{1,\dots,m\}
            \setminus I$ such that $\cS_{j}^+ \cap \triangle_{i \in I}\cS_i^+$ is non-empty.
\end{enumerate}
\label{lem_cs_fog}
\end{lemma}

\begin{proof} 
    We adopt the definitions of $\cS'_i$, $\cS_i$ and $\cS_i^\dagger$ given in the definition of a $\mathbb{Z}_2$-symmetric Fogelsanger definition. 
    Note that, for all $1\leq i\leq m$, $\cS_i \neq \cS$ since $xy \not\in E(\cS_i)$.
    Suppose $F \in \partial \cS_i$. Then $F \cap \{x,y,x^*,y^*\} \neq \emptyset$ otherwise $F \in 
    \partial S$ contradicting the fact that $\cS$ is a simplicial $k$-cycle. For $v \in \{x,y,x^*,y^*\}$
    let 
    $\cU_v = \{U \in \partial \cS_i: v \in U\}$. Since no edge in $E(\cS_i)$ joins two vertices of $\{x,y,x^*,y^*\}$, we have  
    $\partial \cS_i = 
    \cU_x \sqcup \cU_y \sqcup \cU_{x^*} \sqcup \cU_{y^*}$. Furthermore since $\partial (\cS_i/xy)/x^*y^*
    = \partial \cS'_i = \emptyset$ it follows that 
    \begin{equation}
        \label{eqn_props0}
        \text{the mapping $F \mapsto F-y+x$ is a bijection $\cU_y \rightarrow \cU_x$.}
    \end{equation}
    and 
    \begin{equation}
        \label{eqn_props0_star}
        \text{the mapping $F \mapsto F-y^*+x^*$ is a bijection $\cU_{y^*} \rightarrow \cU_{x^*}$.}
    \end{equation}
    In particular
    \begin{equation}
        \label{eqn_props1}
        \cS^\dagger_i = \{F +y: F \in \cU_x\}.
    \end{equation}
    Also $*$ induces a bijection $\cU_x \rightarrow \cU_{x^*}$ and hence
    \begin{equation}
        \label{eqn_props2}
        |\cS^\dagger_i| = |\cU_x|= |\cU_y|= |\cU_{x^*}|= |\cU_{y^*}|.
    \end{equation}
    Now if $xy \not\in \cS_i^+$ then $\cS^\dagger_i =\emptyset$ and it follows from (\ref{eqn_props2})
    that $\partial \cS_i = \emptyset$,
    contradicting the fact that $\cS$ is a \csir{$k$} and $\emptyset \neq \cS_i \subsetneq \cS$.
    So $xy \in E(\cS_i^+)$ and since $\cS_i^+$ is $*$-invariant, $x^*y^* \in E(\cS_i^+)$, 
    proving (\ref{cs-a}).

    The definition of $\cS_i^+$ implies that $(\cS_i^+/xy)/x^*y^* = \cS'_i$. Since $\cS_i'$ is a \csir{$k$}, this gives (\ref{cs-c}).

    \begin{claim}
        \label{clm_boundaries}
        $\partial( \cS_i^\dagger \cup (\cS_i^\dagger)^*) = \partial \cS_i$.
    \end{claim}

    \begin{proof}[Proof of claim.]
        From (\ref{eqn_props2}) and the definition of $\cS_i^\dagger$ it follows that,
        for $0 \leq s \leq k-1$, the map $F \mapsto F+y$ is a bijection between the set of $s$-faces 
        of $\cU_x$ that contain $x$ and the set of $(s+1)$-faces of $\cS^\dagger_i$ that contain $\{x,y\}$. 
        Since $\partial \cS_i$ is a simplicial $(k-1)$-cycle, every $(k-2)$-face of $\cU_x$
        that contains $x$ belongs to an even number of $(k-1)$-faces of $\cU_x$. Therefore every $(k-1)$-face of
        $\cS_i^\dagger$ that contains $\{x,y\}$ belongs to an even number of $k$-faces of $\cS_i^\dagger$. 
        Hence $\partial \cS_i^\dagger = \{K-x: K \in \cS_i^\dagger\} \cup \{K-y: K \in \cS_i^\dagger\}
        = \cU_y \sqcup \cU_x$.  By symmetry, $\partial( (\cS_i^\dagger)^*) = \cU_{y^*} \sqcup \cU_{x^*}$. Therefore 
        $\partial \cS_i = 
    \cU_x \sqcup \cU_y \sqcup \cU_{x^*} \sqcup \cU_{y^*}=\partial( \cS_i^\dagger \cup (\cS_i^\dagger)^*) $.
    \end{proof}

    Since
    $\cS_i^+ = \cS_i \triangle ( \cS_i^\dagger \cup (\cS_i^\dagger)^*)$, Claim \ref{clm_boundaries} implies
    that 
    $\partial \cS_i^+ = \emptyset $ and hence $\cS_i^+$ is a $\Z_2$-symmetric $k$-cycle.

    Suppose $\cT\subseteq \cS_i^+$ is a \csir{$k$}. If $xy, x^*y^* \not \in E(\cT)$ then $\cT \subsetneq \cS$,
    contradicting the fact that $\cS$ is a \csir{$k$}. So $xy,x^*y^* \in E(\cT)$. Now $(\cT/xy)/x^*y^*$ is 
    a non-empty $\Z_2$-symmetric $k$-cycle contained in the \csir{$k$} $\cS'_i$. Hence $(\cT/xy)/x^*y^*=\cS_i'$ and $\cS_i\subseteq \cT$.
    Since $\partial \cT  = \emptyset $ it follows that $\cT \supset \cS_i^\dagger$ and so 
    $\cT = \cS_i^+$. Hence $\cS_i^+$ is a \csir{$k$}, which is nontrivial  by (\ref{cs-c}).
    Since $\cS_i^+\setminus \cS \subset \cS_i^\dagger \cup (\cS_i^\dagger)^*$, each $K \in \cS_i^+\setminus \cS$ is a clique of $G$ that either contains $\{x,y\}$ or contains $\{x^*,y^*\}$. This completes the proof of (\ref{cs-d}).

    If $U\in \cS$ contains neither $\{x,y\}$ nor $\{x^*,y^*\}$ then $U$
    is an element of $\cS_i$ for a unique $i$, proving (\ref{cs-e}).

    We next prove (\ref{cs-f}). 
Suppose, for a contradiction, that 
    $\cR= \cS \triangle \left(\triangle_{i=1}^m \cS_i^+\right)\neq \emptyset$ and choose $U \in \cR$. By 
    (\ref{cs-e}), 
    either $\{x,y\} \subset U$ or $\{x^*,y^*\} \subset U$. Suppose $\{x,y\} \subset U$ (a similar argument 
    works for the other case).
    Now $\cR$ is a simplicial $k$-cycle and since $\cS$ and $\cS_i^+, 1\leq i\leq m$, are all
    nontrivial \csir{$k$}s, $\cR$ does not contain any trivial simplicial $k$-circuit. 
    It follows that there is 
    some $U' \in\cR$ such that $U \cap U' = U -x$. By (\ref{cs-e}), $U'$ must contain $\{x^*,y^*\}$ and so
    $yy^* \in E(\cR) \subset E(\cS)$, contradicting the assumption that $*$ is a simplicial involution on $\cS$. Therefore $\cR$ must be empty - this proves
    (\ref{cs-f}). 

    To verify (\ref{cs-g}) we suppose, for a contradiction, that there exists a proper subset $I$ of $\{1,\dots,m\}$ 
    such that  $\cS_j^+ \cap \triangle_{i \in I} \cS_i^+   = \emptyset $ for all $j \in J = \{1,\dots,m\} \setminus I$.
    Then, by (\ref{cs-f}), $\cS = \triangle_{i=1}^m \cS_i^+ = \triangle_{i \in I} \cS_i^+ \sqcup \triangle_{j \in J} \cS_j^+$. 
    This yields a partition of $\cS$ into two $\Z_2$-symmetric $k$-cycles, contradicting the hypothesis that 
    $\cS$ is a \csir{$k$}. This proves (\ref{cs-g}).
    
    We can now complete the proof of the lemma by verifying  (\ref{cs-b}). Observe that, if $uv \in E(\cS)$ and $uv\not\in \{xy,x^*y^*\}$ then, since $\cS$ is a nontrivial  \csir{$k$} and $xy^*,x^*y\not\in E(\cS)$, there is some $U \in \cS$ such that $u,v \in U$ and $\{x,y\}, \{x^*,y^*\} \not\subset U$. 
    By (\ref{cs-e}), $uv \in E(\cS_i^+)$ for some $i$. Using this, together with (\ref{cs-a}), we see that $\bigcup_{i = 1}^m E(\cS_i^+)
    = E$. Since $G(\cS)$ has no isolated vertices, this implies that $\bigcup_{i=1}^m V(\cS_i^+) = V$. 
    This proves (\ref{cs-b}).
\end{proof}

We need one more additional property of $\mathbb{Z}_2$-symmetric Fogelsanger decompositions.
\begin{lemma}\label{lem:h}
Let $(\cS,*)$ be a \csir{$k$} and $xy\in E(\cS)$  with $xy^*\not\in E(\cS)$. Let $\{\cS_1',\cS_2',\ldots,\cS_m'\}$ be a partition of $(\cS/xy)/x^*y^*$ into \csir{$k$}s and $(\cS_1^+, \dots, \cS_m^+)$ be the $\mathbb{Z}_2$-symmetric Fogelsanger decomposition for $\cS$ at $xy$ corresponding to this partition.
Suppose  $\cS_j'=\cT_j'\sqcup \cT_j'^*$ for some simplicial $k$-circuit $\cT_j'$ with $|V(\cT_j')\cap V(\cT_j'^*)|\leq 2k-2$.
Then $\cS_j^+=\cT_j\sqcup \cT_j^*$ for some simplicial $k$-circuit $\cT_j$ with $|V(\cT_j)\cap V(\cT_j^*)|\leq |V(\cT_j')\cap V(\cT_j'^*)|+2$.
\end{lemma}
\begin{proof}
    Since $\cS_j'=(\cS_j/xy)/x^*y^*$, the hypothesis that
    $|V(\cT') \cap V(\cT'^*)| \leq 2k-2$ implies that
    $G(\cS_j^+)$ has a $*$-invariant vertex separator $X$ with $V(\cT') \cap V(\cT'^*)\subseteq X\subseteq V(\cT') \cap V(\cT'^*)+y+y^*$. 
    In particular $|X|\leq 2k$.
    By Lemma~\ref{lem_sep_2k}, either 
    $\cS_j^+=\cT_j\sqcup \cT_j^*$ for some simplicial $k$-circuit $\cT_j$ with $V(\cT_j)\cap V(\cT_j^*)=X$ or 
    $G(\cS_j^+)[X]$ is isomorphic to $G(\cB_{k-1})$.
    In the former case, we have the statement.
    So let us assume the latter case.
    Then we have $|X|=2k$ and hence $|X|=|V(\cT') \cap V(\cT'^*)|+2$, so $\{x,y,x^*,y^*\}\subseteq X$.
    Together with $G[X] \cong G(\cB_{k-1})$, this implies that 
    $xy^* \in E$, contradicting an hypothesis of the lemma.
\end{proof}

\subsection{Main rigidity theorem}
We can now establish our main result.
We begin with the following special case.
\begin{lemma}\label{lem:special}
Let $\Gamma=\Gamma_{t,k+1}$ for $0 \leq t \leq k$.
Let $(\cS,\ast)$ be a $\mathbb{Z}_2$-irreducible $k$-cycle such that 
$\cS=\cT\sqcup \cT^*$ for some $k$-circuit $\cT$ with $h:=|V(\cT)\cap V(\cT^*)|\leq 2k$.
Then $(G(\cS),\ast)$ is $\Gamma$-rigid in $\R^{k+1}$ if and only if $h\geq \max\{2(k-t),k+1,2t\}$.
\end{lemma}
\begin{proof}
Suppose $h \geq \max\{2(k-t),k+1,2t\}$.
Let $*_T$ be the non-adjacent vertex pairing induced on $V(\cT)\cap V(\cT^*)$ by $*$.
Since $h\leq 2k$, we can apply Theorem~\ref{thm_smallX} to $(G(\cT),\ast_T)$ to deduce that $(G(\cT),\ast_T)$ is $\Gamma$-rigid.
Symmetrically, $(G(\cT^*),\ast_T)$ is $\Gamma$-rigid.
Now Theorem~\ref{thm_pcs_gluing} (the gluing theorem) implies that the union of 
$(G(\cT),\ast_T)$ and $(G(\cT^*),\ast_T)$ is $\Gamma$-rigid if $h\geq (k+1)+h-1-\min\{h/2,d-t\}-\min\{h/2-1,t\}$,
or equivalently $\min\{h/2,k+1-t\}+\min\{h/2-1,t\}\geq k$.
By an elementary calculation, one can check that this is equivalent to $h\geq \max\{2(k-t),k+1,2t\}$.

On the other hand, suppose $h < \max\{2(k-t),k+1,2t\}$ and let $(G(\cS),\ast,p)$ be a generic $\Gamma$-symmetric framework. If $\cT$ and 
$\cT^*$ are trivial simplicial $k$-circuits then, by Lemma \ref{lem_cs_circuits},
$V(\cT) \cap V(\cT^*) = \emptyset$ and it is clear that $G(\cS)$ is not $\Gamma$-rigid. So we may assume that 
$\cT$ and $\cT^*$ are nontrivial simplicial $k$-circuits. It follows easily that 
$$\dim(\aff(p(V(\cT)))) = \dim(\aff(p(V(\cT^*)))) = k+1.$$ However the assumption 
$h < \max\{2(k-t),k+1,2t\}$ implies that $H = \aff(p(V(\cT)\cap V(\cT^*)))$ has 
dimension at most $k-1$ and so $(G(\cT),\ast_T,p)$ has a nontrivial infinitesimal flex induced by rotation about $H$.
\end{proof}

\begin{theorem}
    Let $(\cS,*)$ be a nontrivial \csir{$k$} for some $k \geq 2$. Suppose that  $\Gamma=\Gamma_{t,k+1}$ for some $0\leq t \leq k$ and that $\Gamma$ is not a rotation group when $k=2$.
    Then the following statements are equivalent.
    \begin{itemize}
        \item[(i)] $(G(\cS),\ast)$ is $\Gamma$-rigid in $\R^{k+1}$;
        \item[(ii)] $\cS$ is a simplicial $k$-circuit or $\cS=\cT\sqcup \cT^*$ for some simplicial $k$-circuit $\cT$ with $|V(\cT)\cap V(\cT^*)|\geq \max\{2(k-t),k+1,2t\}$;
        \item[(iii)] for every $X \subset V(\cS)$ such that
    $X^* = X$ and $|X| < \max\{2(k-t),k+1,2t\}$, 
    $G(\cS) - X$ is connected.
    \end{itemize}
    \label{thm_real_main}
\end{theorem}

\begin{proof}
Let $c=\max\{2(k-t),k+1,2t\}$.

By Lemmas~\ref{lem_cs_circuits} and \ref{lem:special}, (i) implies (ii).
To see that  (ii) and (iii) are equivalent, observe that 
$c\leq 2k$ since $0 \leq t\leq k$.
Also, if $X$ is a $*$-invariant subset of $V(\cT)$ of size less than $2k$, then Lemma~\ref{lem_sep_2k} implies that $X$ is a vertex separator of $G(\cT)$ if and only if $\cS = \cT \sqcup \cT^*$ for some simplicial $k$-circuit $\cT$ with $V(\cT) \cap V(\cT^*) = X$.
The equivalence between (ii) and (iii) now follows.

It remains to prove that (ii) and (iii) imply (i).
    Suppose, for a contradiction, that $\cS$ is a counterexample to this statement with $|V(\cS)|$ minimal and 
    let $G(\cS)= G =(V,E) $.

    \begin{claim}
        \label{clm:edge}
        There exists an edge $xy \in E$ such that $x^*y \not\in E$.
    \end{claim}
    \begin{proof}[Proof of claim.]
Suppose, for a contradiction, that $x^*y \in E$ whenever  $xy \in E$. Then, for every $U \in \cS$, 
$G[U \sqcup U^*]$ can be obtained  from the complete graph on $U\sqcup U^*$ by deleting the set of edges $\{xx^*:x\in U\}$. Hence $G[U \sqcup U^*]$ is isomorphic to $G(\cB_k)$.  We can now apply Lemma~\ref{lem:cross} to deduce that 
    \begin{equation} 
        \label{eqn_inducedcross} 
        \text{$G[U \sqcup U^*]$  is $\Gamma$-rigid for all $U \in \cS$.}
    \end{equation}

    Suppose $\cS$ is a simplicial $k$-circuit.  Then $\cS$ is a strongly connected simplicial $k$-complex so, for all $U,U'\in \cS$, there exists a sequence of $k$-faces  $U_1,U_2,\ldots,U_m$ of $\cS$ such that $U_1=U$, $U_m=U'$ and $|U_i\cap U_{i+1}|=k$ for all $1\leq i\leq m-1$.
    Putting $H=\bigcup_{i=1}^mG[U_i \cup U_i^*]$, we can now use  
        (\ref{eqn_inducedcross}) and Theorem \ref{thm_pcs_gluing} to deduce that $(H,p|_H)$ is infinitesimally rigid.   This implies that every vertex of $G$ belongs to the unique maximal $\Gamma$-rigid subgraph of $G$ which contains $U\sqcup U^*$ and hence $G$ is  $\Gamma$-rigid. This contradicts the choice of $\cS$.

    Thus  $\cS$ is not a simplicial $k$-circuit. Then
$\cS = \cT \sqcup \cT^*$ for some simplicial $k$-circuit $\cT$ with  $|V(\cT) \cap V(\cT^*)| \geq c$ by the assumption that (ii) holds for $\cS$. Then $\cT$ is strongly connected so we can apply the argument in the previous paragraph to deduce that  $V(\cT)$  is contained in a $\Gamma$-rigid subgraph of $G$. 
By symmetry,  $V(\cT^*)$  is also contained in a $\Gamma$-rigid subgraph of $G$.   Since $V(G) = V(\cT) \cup V(\cT)^*$,  we can now apply Theorem~\ref{thm_pcs_gluing} to deduce that $G$ is $\Gamma$-rigid, again a contradiction.
    This completes the proof of the claim.
\end{proof}

    Claim \ref{clm:edge} tells us we can choose an edge $xy\in E$ such that $xy^*\not\in E$. 
Let $(\cS_1^+, \dots, \cS_m^+)$ be a $\mathbb{Z}_2$-symmetric Fogelsanger decomposition of $\cS$ at $xy$, and put $\cS_i' = (\cS_i^+/xy)/x^*y^*$ for all $1\leq i\leq m$.
    By Lemma~\ref{lem_cs_fog}(\ref{cs-c}), $\cS_i'$ is a $\mathbb{Z}_2$-irreducible cycle.

    \begin{claim}\label{claim_2}
        For all $1\leq i\leq m$, either $G(\cS_i^+)$ is $\Gamma$-rigid or $\cS_i^+=\cT_i\sqcup \cT_i^*$ for some simplicial $k$-circuit $\cT_i$ with $|V(\cT_i)\cap V(\cT_i^*)|\leq 2k$.
    \end{claim}
    \begin{proof}
 The proof splits into two cases depending on the $\Gamma$-rigidity of  $\cS_i' = (\cS_i^+/xy)/x^*y^*$.

        Suppose $\cS_i'$ is $\Gamma$-rigid.
        Since $xy \in E(\cS_i^+)$  and $\cS_i^+$ is a nontrivial \csir{$k$} by Lemma~\ref{lem_cs_fog}(a)(d), there exist 
$U_1,U_2\in \cS_i^+$ such that $|U_1 \cap U_2| = k$ and $\{x,y\} \subseteq U_1 \cap U_2$. Let 
$C = (U_1 \cup U_2)\setminus \{x,y\}$. Then  $C \subset N_{G(\cS_i^+)}(x) \cap N_{G(\cS_i^+)}(y)$, $|C|=k$ and $|C \cap C^*| \leq 2$. Similarly we can choose $D \subset N_{G(\cS_i^+)}(x^*) \cap N_{G(\cS_i^+)}(y^*)$ such that $|D| = k$ and 
$|D \cap D^*| \leq 2$. 
Theorem \ref{thm_whiteley_pcs_contraction} now implies that $G(\cS_i^+)$ is $\Gamma$-rigid.

Next suppose $\cS_i'$ is not $\Gamma$-rigid.
Then, by the minimality of $|V(\cS)|$, $\cS_i'$ does not satisfy (ii) in the statement of the theorem. Lemma~\ref{lem_cs_circuits} now implies that $\cS_i'=\cT'\sqcup \cT'^*$ for some $k$-circuit $\cT'$ with $|V(\cT')\cap V(\cT'^*)|<c$.
By Lemma~\ref{lem:h}, there is a simplicial $k$-circuit $\cT$ such that 
$\cS_i^+=\cT\sqcup \cT^*$ with $|V(\cT)\cap V(\cT^*)|\leq |V(\cT')\cap V(\cT'^*)|+2$.
Suppose $1 \leq t \leq k-1$. Then $c=\max\{2(k-t),k+1,2t\}\leq 2k-1$. Hence $|V(\cT)\cap V(\cT^*)|\leq |V(\cT')\cap V(\cT'^*)|+2\leq  c+1\leq 2k$, as required.
Hence we may assume that $t=0$ or $t=k$. Then $c=2k$. Since $|V(\cT')\cap V(\cT'^*)|$ and $c$ are both even, 
$|V(\cT)\cap V(\cT^*)|\leq |V(\cT')\cap V(\cT'^*)|+2\leq c=2k$.
This completes the proof.
    \end{proof}

We next define a sequence of simplicial $k$-cycles $\cW_1,\cW_2,\ldots, \cW_m$.
For each $1\leq i\leq m$, we set $\cW_i=\cS_i^+$ if  
$G(\cS_i^+)$ is $\Gamma$-rigid. Otherwise  we put $\cW_i=\cT_i$, where $\cT_i$ is the simplicial $k$-circuit given by Claim~\ref{claim_2} for ${\cal S}_i^+$.
We will show that
\begin{equation}\label{eq:rigidity}
\text{$G(\cW_i)$ and $G(\cW_i^*)$ are both $\Gamma$-rigid.}
\end{equation}
If $\cW_i=\cS_i^+$ then $G(\cW_i)$ is $\Gamma$-rigid by definition and we have $\cW_i^*=(\cS_i^+)^*=\cS_i^+$ because $\cS_i^+ $ is a \csir{$k$}.
On the other hand, if $\cW_i=\cT_i$, then we have 
$|V(\cT_i) \cap V(\cT_i^*)| \leq 2k$ and we can apply Theorem \ref{thm_smallX} to deduce that $G(\cT_i)$ and $G(\cT_i^*)$ are both $\Gamma$-rigid. Thus (\ref{eq:rigidity}) holds.

    We can now complete the proof of the theorem.
    Using Lemma \ref{lem_cs_fog}(\ref{cs-g}) to reorder $(\cS_1^+,\ldots,\cS_m^+)$ as necessary, and interchanging 
    $\cT_i^*$ and $\cT_i$ where necessary, we can assume that 
   \begin{equation}
        \label{eqn_nonempty}
        \cW_j \cap (\triangle_{i=1}^{j-1}\cW_i) \neq \emptyset \mbox{ for $2\leq j \leq m$.}
    \end{equation}
The hypothesis that $xx^*\not\in E$ for all $x\in X$ implies that 
\begin{equation}\label{eqn:empty}
\mbox{$U\cap U^*=\emptyset$ for all cliques $U$ of $G(\cS)$.}
\end{equation} 
We can now apply a straightforward induction argument using  (\ref{eq:rigidity}), (\ref{eqn_nonempty}), (\ref{eqn:empty})
    and Theorem \ref{thm_pcs_gluing} to deduce that 
    \begin{equation}
        \label{eqn_rig}
        \text{ 
            $\bigcup_{i=1}^m G(\cW_i)$ and $\bigcup_{i=1}^m G(\cW^*_i)$ are both
        $\Gamma$-rigid.}
    \end{equation}
    
If $\cW_i=\cS_i^+$ for some $1\leq i\leq m$, then  
$\bigcup_{i=1}^m G(\cW_i)$ and $\bigcup_{i=1}^m G(\cW^*_i)$ 
intersect in  at least $2k$ vertices. By Theorem~\ref{thm_pcs_gluing},
$G(\cS) = \left(\bigcup_{i=1}^m G(\cW_i)\right) \cup \left(\bigcup_{i=1}^m G(\cW^*_i)\right)$ is $\Gamma$-rigid
contradicting our choice of $\cS$.

So we can assume that for $i = 1,\dots,m$, $\cS_i^+ = \cT_i \sqcup \cT_i^*$ 
where $\cT_i$ is simplicial $k$-circuit properly contained in 
$\cS_i^+$ and $\cW_i = \cT_i$. If  $\cT_i\cap \cT_j^*\neq \emptyset$ for some distinct $i,j$, 
then $\bigcup_{i=1}^m G(\cT_i)$ and $\bigcup_{i=1}^m G(\cT^*_i)$ intersect on a $(k+1)$-clique in $G(\cS)$. Hence their union is $\Gamma$-rigid by (\ref{eqn:empty}) and the gluing theorem (Theorem~\ref{thm_pcs_gluing}). Since the union is $G(\cS)$ by Lemma~\ref{lem_cs_fog}(b), this contradicts the choice of $\cS$ as a counterexample.

Therefore $\cT_i\cap \cT_j^*=\emptyset$ for all $1\leq i, j\leq m$.
By Lemma~\ref{lem_cs_fog}(e), $\cT_i^*$ contains a simplex $U\in \cS$ that is not contained in any $\cS_j^+$ for $j\neq i$.
This, and the facts that $\cT_i\cap \cT_j^*=\emptyset$ for all $1\leq i,j\leq m$ and $\triangle_{i=1}^{m}\cS_i^+=\cS$, imply that $\triangle_{i=1}^{m}\cT_i$ is a simplicial $k$-cycle that is properly contained in $\cS$, so $\cS$ is not a simplicial $k$-circuit.
Since $\cS$ satisfies (ii) from the statement of the theorem, it follows from Lemma 5.1 that $\cS=\cT  \sqcup \cT^*$, where $\cT=\triangle_{i=1}\cT_i$ and
$|V(\cT)\cap V(\cT^*)|\geq \max\{2(k-t),k+1,2t\} =:c$.
Moreover, also by Lemma \ref{lem_cs_circuits}, $\cT$ must be a simplicial $k$-circuit.

Note that $G(\cT)\subseteq \bigcup_{i=1}^m G(\cT_i)$ and $G(\cT^*)\subseteq \bigcup_{i=1}^m G(\cT_i^*)$,
so $\bigcup_{i=1}^m G(\cT_i)$ and $\bigcup_{i=1}^m G(\cT_i^*)$ also have at least $c$ common vertices.
We can now use
(\ref{eqn_rig}), Theorem \ref{thm_pcs_gluing}, and Lemma~\ref{lem_cs_fog}(b)  to deduce that $G(\cS)$ is $\Gamma$-rigid, contradicting 
our choice of $\cS$. This final contradiction completes the proof of the theorem.
\end{proof}

Theorem \ref{thm_real_main} immediately gives the following sufficient condition for the $\Gamma$-rigidity of $\Z_2$-symmetric  simplicial circuits.

\begin{theorem}
    \label{cor_main_rigid}
    Let $(\cS,*)$ be a $\Z_2$-symmetric simplicial $k$-circuit
with $k\geq 2$. Then $(\cS,*)$ is $\Gamma$-rigid in $\R^{k+1}$ if either $k\geq 3$ or $k=2$ and $\Gamma$ is not the half-turn rotation group.
\end{theorem}

Since every pseudomanifold is a simplicial circuit, \cite[Conjecture 8.3]{KNNZ} follows immediately as a special case of Theorem \ref{cor_main_rigid}.

\section{The Lower Bound Theorem}
\label{sec_lbt}

The lower bound theorem for \csir{$k$}s is an immediate corollary of Theorem \ref{thm_real_main} and Lemma~\ref{lem:symmetric_maxwell}.

\begin{theorem}
    Let $(\cS,*)$ be a nontrivial simplicial \csir{$k$} for some $k \geq 2$.
    Suppose that $G(\cS) - X$ is connected for all $X \subset V$ which satisfy $X^* = X$ and 
    $|X| \leq 2k$. Then $g_2( \cS) = |E(\cS)| - (k+1)|V(\cS)| + \binom{k+2}2 \geq \binom{k+1}2 - (k+1)$.
    \label{thm_KNNZ_conj}
\end{theorem}

\begin{proof}
We can apply  Theorem \ref{thm_real_main} in the case when $\Gamma = \Gamma_{0,k+1}$ is a point inversion to deduce that $(G(\cS),*)$ is $\Gamma$-rigid. The case $t=0$ of   Lemma~\ref{lem:symmetric_maxwell} now gives $g_2( \cS) \geq \binom{k+1}2 - (k+1)$.
\end{proof}

A similar argument using Theorem \ref{cor_main_rigid} and Lemma~\ref{lem:symmetric_maxwell} gives
\begin{theorem}
    \label{cor_main}
    Let $(\cS,*)$ be a $\Z_2$-symmetric simplicial $k$-circuit
with $k\geq 2$. Then $g_2(\cS) \geq \binom{k+1}2 -(k+1)$.
\end{theorem}

Since every pseudomanifold is a simplicial circuit, this immediately gives the following lower bound result for pseudomanifolds, and hence verifies the inequality part of \cite[Conjecture 8.1]{KNNZ} 

\begin{corollary}
    \label{cor_main_2}
    Let $(\cS,*)$ be a  $\Z_2$-symmetric 
    $k$-pseudomanifold 
    with $k \geq 2$. Then $g_2(\cS) \geq \binom{k+1}2 - (k+1)$.
\end{corollary}

\section{Closing Remarks}\label{sec_conclusion}

\subsection{Characterising $\Gamma_{1,3}$-rigid simplicial $2$-circuits}

Theorem \ref{thm_real_main} resolves the $\Gamma_{t,k+1}$-rigidity 
question for $\Z_2$-symmetric simplicial $k$-circuits in $\R^{k+1}$, except in the case when $t=1$ and  $k=2$.
In this case,  Lemma~\ref{lem:symmetric_maxwell} implies that the boundary complex of a $\Z_2$-symmetric simplicial polyhedron cannot be 
$\Gamma_{1,3}$-rigid in $\R^3$. Other examples  can be obtained from the boundary complex $\cP$ of a $\Z_2$-symmetric simplicial polyhedron by choosing a face $F$ of $\cP$ and `inserting'  an arbitrary simplicial $2$-circuit $\cT$ into both $F$ and $F^*$.  More precisely we chose a  copy $\cT^*$  of $\cT$, label the vertices of $\cT$ so that $F\in \cT$ and  $V(\cP)\cap V(\cT)=F$,  and then put $\cS=\cP\triangle (\cT\sqcup \cT^*)$.

We believe that all examples of  $\Z_2$-symmetric simplicial $2$-circuits which are not $\Gamma_{1,3}$-rigid in $\R^{3}$ can be obtained by recursively applying this insertion operation to the boundary complex of a $Z_2$-symmetric simplicial polyhedron.
This would imply the following conjecture.

\begin{conjecture}
    \label{conj_2_circuits}
    Suppose that $(\cS,*)$ is a $\mathbb Z_2$-symmetric simplicial
    2-circuit such that $G(\cS)$ is 4-connected and non-planar. Then $G(\cS)$ is $\Gamma_{1,3}$-rigid in $\R^3$.
\end{conjecture}

\subsection{Redundant $\Gamma$-rigidity of simplicial circuits}
A graph $G=(V,E)$ is {\em redundantly rigid} in $\R^d$ if $G-e$ is rigid in $\R^d$ for all $e\in E$. We showed in \cite{CJT1} that, if $G$ is the graph of a simplicial $k$-circuit for some $k\geq 3$ and $G$ is $(k+1)$-connected,  then $G$ is redundantly rigid in $\R^{k+1}$, and used this to deduce that equality in the Lower Bound Theorem can only hold for simplicial $k$-circuits when they are stacked $k$-spheres. We believe that an analogous approach may work for $\Gamma$-rigidity and our symmetric lower bound theorem.

\begin{conjecture}
    \label{con:red_rigid}
    Let $(\cS,*)$ be a $\Z_2$-symmetric simplicial $k$-circuit 
    such that $k\geq 3$, $\cS\neq \cB_k$ and $G(\cS)$ is $(k+1)$-connected. Then, for all point groups $\Gamma$ in $\R^{k+1}$ of order two and all $e\in E(\cS)$, $(G(\cS)-e-e^*,*)$ is $\Gamma$-rigid in $\R^{k+1}$.
\end{conjecture}
Conjecture \ref{con:red_rigid} would imply that equality can only hold  in Theorem \ref{cor_main} when $\cS$ is a symmetrically stacked sphere (see \cite{KNNZ} for a definition)  and would verify \cite[Conjecture 8.1]{KNNZ} as a special case.

\subsection{Non-adjacent vertex pairings}
The Bricard octahedron shows that
Theorem \ref{thm_smallX} will become false if we remove the hypothesis that 
$|X| \leq 2k$. More generally, any $\Z_2$-symmetric  boundary complex of a polyhedron will not be $\Gamma$-rigid in $\R^3$ when $\Gamma$ is the half turn rotation group by Lemma~\ref{lem:symmetric_maxwell}.  
It is conceivable, however, that the hypothesis 
that $|X| \leq 2k$ 
 is not necessary in certain special cases. 
Indeed \cite[Conjecture 8.4]{KNNZ} is precisely this 
statement in the case when $k=2$, $\cS$ is a simplicial sphere and $\Gamma$ is a point inversion group.  

\subsection{More general group actions}
Adin~\cite{A95}, and later Jorge~\cite{J03}, extended Stanley's lower bound theorem to rational polytopes with a cyclic linear symmetry group of prime power order. Their methods are very different from ours, relying on 
algebraic properties of the Stanley-Reisner ring of a simplicial complex.

It is natural to ask if the rigidity-based approach can be used to extend these results to simplicial circuits and/or other point groups.
In this paper we frequently use the fact that $*$ is free and $\Gamma$ is a point group of order two. Relaxing either of these restrictions in a rigidity-based approach remains an open problem.


\subsection{The lower bound conjecture for higher dimensional faces}

Let $\phi_j(n,k)$ be the number of $j$-faces of a symmetrically stacked
$k$-sphere with $n$ vertices. One can easily verify by induction on $n$ that 
\[
    \phi_j(n,k) = \left\{ \begin{array}{ll} 
    2^{j+1}\binom{k+1}{j+1} + (n-2k-2)\binom{k+1}{j} & \text{ if } j \leq k-1 \\
    2^{k+1}+ k(n-2k-2) & \text{ if }j =k
    \end{array}\right.
\]

For a simplicial complex $\cS$ let $f_j(\cS)$ be the number of $j$-faces of
$\cS$.

\begin{conjecture}
    \label{conj_higher}
    Suppose that $k\geq 2$ and $\cS$ is 
    a simplicial $k$-circuit with $n$ vertices and $*$ is a free simplicial 
    involution on $\cS$. Then $f_j(\cS) \geq \phi_n(n,k)$ for $j = 1,\dots,k$.
    Moreover, if equality occurs for any $j$ then $\cS$ must be a $k$-sphere and,
    in the case that $k\geq 3$, must be a symmetrically stacked $k$-sphere.
\end{conjecture}

In the non-symmetric setting the analogous conjecture for simplicial manifolds
can be derived from 
the case $j=1$ by an induction argument based on the fact that the link of
any proper face in a simplicial manifold is a simplicial sphere of lower dimension (see \cite{kalai} for details). However in the $\mathbb Z_2$-symmetric case this argument does not generalise easily, since the 
the link of $F$ is not $\mathbb Z_2$-symmetric.
One might consider the union of the links of $F$ and $F^*$
but that is not necessarily a simplicial manifold, even when $\cS$ is a simplicial sphere. 

A further difficulty for Conjecture \ref{conj_higher}, and indeed for the 
non-symmetric version \cite[Conjecture 8.]{CJT1}, arises  from the fact that the class of simplicial circuits is not link closed. In general the link of a $j$-face in a simplicial $k$-circuit will be a simplicial $(k-j-1)$-cycle but not necessarily a simplicial $(k-j-1)$-circuit.

\printbibliography

\end{document}